\documentclass[11pt,reqno]{amsart}
\usepackage{amsfonts}
\usepackage{mathtools}
\usepackage{amsmath}
\usepackage{amsthm}
\usepackage{amssymb}
\usepackage{comment}
\usepackage{amscd}
\usepackage{cite}
\usepackage[ngerman,english]{babel}
\usepackage[mathscr]{eucal}
\usepackage{indentfirst}
\usepackage{graphicx}
\usepackage{graphics}
\usepackage{pict2e}
\usepackage{epic}
\numberwithin{equation}{section}
\usepackage[margin = 3.5cm, top = 3cm, bottom = 3cm]{geometry}
\usepackage{epstopdf} 
\usepackage{bbm}
\usepackage{xcolor}
\usepackage{comment}
\usepackage{tikz}
\usepackage{mathtools}
\usepackage{enumitem}

\let\OLDthebibliography\thebibliography
\renewcommand\thebibliography[1]{
	\OLDthebibliography{#1}
	\setlength{\parskip}{0pt}
	\setlength{\itemsep}{3pt}
}
\definecolor{myblue}{HTML}{26649b}
\definecolor{mygreen}{HTML}{46C755}

\calclayout

\usepackage{hyperref}
\hypersetup{ colorlinks=true,
    allcolors= [RGB]{0,0,128} 
}

\numberwithin{equation}{section}

\theoremstyle{plain}
\newtheorem{Th}{Theorem}[section]
\newtheorem{Lemma}[Th]{Lemma}
\newtheorem{lem}[Th]{Lemma}
\newtheorem{defn}[Th]{Definition}
\newtheorem{Cor}[Th]{Corollary}
\newtheorem{Prop}[Th]{Proposition}
\newtheorem{prop}[Th]{Proposition}
\DeclareMathOperator{\R}{\mathbb{R}}
\DeclareMathOperator{\Z}{\mathbb{Z}}
\DeclareMathOperator{\C}{\mathbb{C}}
\DeclareMathOperator{\N}{\mathbb{N}}

\DeclareMathOperator{\cN}{\mathcal{N}}

\newcommand{\f}[2]{\frac{#1}{#2}}

\def \re{\mathrm{Re}}
\def \imm{\mathrm{Im}}
\def \d{\mathrm{d}}
\def \dt{\mathrm{d}t}
\def \G{\mathcal{G}}
\def \H{\mathcal{H}}

\theoremstyle{definition}

\newtheorem{Rem}[Th]{Remark}
\newtheorem{?}[Th]{Problem}

\renewcommand{\Re}{Re}

\setlist[enumerate]{leftmargin=10mm}

\begin{document}

\title[Soliton clusters for the $L^2$-critical Hartree equation]{On soliton clusters and collision blow up\\ for the $L^2$-critical Hartree equation}

	\author{Tobias Schmid}
\address{University of Vienna, Faculty of Mathematics, Oskar-Morgenstern-Platz 1, 1090 Vienna} 
 \email{tobias.schmid@univie.ac.at}
\author{Yutong Wu}
\address{Department of Mathematics, Yale University, New Haven, CT 06511}
\email{yutong.wu.yw894@yale.edu}



\begin{abstract}
We consider the $L^2$-critical nonlinear Hartree equation in $\R^{1+4}$ and multisoliton solutions for which the trajectories are approximated to leading order by an $m$-body law. We obtain soliton clusters asymptotically following hyperbolic-parabolic trajectories of the corresponding $m$-body problem. By pseudo-conformal invariance, we then conclude finite-time collision blow-up with any number of clusters, each consisting of an arbitrary number of solitons, colliding simultaneously at distinct prescribed points. 
\end{abstract}

\maketitle

\section{Introduction}

\subsection{Setting of the problem} We consider the $L^2$-critical nonlinear Hartree equation, i.e.  we study the following cubic Schr\"odinger equation 
	\begin{align}\label{H} 
		\begin{cases}
			\;\; i \partial_t u + \Delta u - \phi_{|u|^2} u = 0,\;\;(t,x) \in [0,T)\times\R^{ 4}&\\[2pt]
			\;\; u(0, x) = u_0(x),\;\;x \in \R^4, &
		\end{cases}
	\end{align}
    where we set $\phi_{|u|^2}$ to be the convolution of $|u|^2$ with the inverse square potential, i.e. 
	\begin{align} \label{def-conv}
		\phi_{|u|^2} = - \frac{1}{|x|^2} \ast |u|^2,\;\; x \in \R^4,
	\end{align}
	and therefore $\Delta \Phi = |u|^2$ for  $ \Phi = \kappa \phi_{|u|^2}$ with a constant $\kappa = \frac{1}{4\pi^2}$. Hence the nonlinearity in \eqref{H} models a gravitational coupling which introduces long-range effects and typically  appears for instance related to the mean field dynamics of Bose gases, cf. \cite{Fr-Lenz} and Boson stars, see \cite{El-Schlein}, \cite{LenzmannNondegen}.\\[4pt]
     It is well known, see \cite{Caz,Ginibre-Velo}, that \eqref{H} is locally wellposed in $ H^1(\R^4)$.  In particular \eqref{H} is a Hamiltonian equation for which $H^1$-solutions conserve \emph{energy, mass} and \emph{momentum}
	\begin{gather}
		\text{\emph{Hamiltonian\;energy}}:\;\;\hspace{0cm}\mathcal{E}(u(t)) = \f12 \int |\nabla u(t)|^2 \;\d x - \frac{\kappa}{4} \int | \nabla \phi_{|u(t)|^2}|^2 \; \d x  = \mathcal{E}(u_0),\\[3pt]
		\text{\emph{$L^2$-norm (mass)}}:\;\; \hspace{0cm}\int |u(t)|^2 \;\d x = \int |u_0|^2\;\d x,\\[3pt]
		\text{\emph{Momentum}}:\;\; \hspace{0cm}\text{Im}\bigg( \int  \nabla u(t) \cdot \overline{u(t)}\;\d x \bigg) =  \text{Im}\bigg( \int \nabla u_0 \cdot \overline{u_0}\;\d x \bigg).
	\end{gather}
	 The equation \eqref{H} has the following symmetries
	\begin{align} \label{inv}
		u(t,x) \mapsto \lambda^2 u( \lambda^2 t + t_0, \lambda x  - \alpha - \beta t) e^{i(\f12 \beta \cdot x - \f14 |\beta|^2 t + \gamma)},
	\end{align}
    where $ \lambda \in \R_+, t_0 \in \R, \alpha, \beta \in \R^4, \gamma \in \R$, meaning if  $ u(t,x) $ is a solution of \eqref{H}, then $u_{\lambda, t_0, \alpha, \beta, \gamma}(t,x)$ defined via \eqref{inv}
	is a solution of \eqref{H}. \\
    [3pt]
    Further if $u \in C^0([0,T), H^1(\R^4))$ is a solution of  \eqref{H} and blows up at $t=T$, then
    \[
    \limsup_{t \to T^-} \| \nabla u \|_{L_x^2}^2 = \infty.
    \]
    In fact, blow-up with $ \mathcal{E}(u_0)< 0$ follows from the virial law 
    \begin{align*} \nonumber
        &\frac{\d^2}{\d t^2} \int |x|^2|u(t,x)|^2\;\d x	 = 16 \mathcal{E}(u_0),
	\end{align*}
    hence all solutions in $H^1 \cap L^2(|x|^2 \d x)$ blow up in finite time if $ \mathcal{E}(u_0)< 0$.\\[3pt] 
	Moreover, equation \eqref{H} has \emph{soliton} solutions $u_{\omega}(t,x)= e^{i \omega^2 t} Q_{\omega}(x)$, where $ Q_{\omega} = \omega^2 Q_1(\omega \cdot) $ and $Q := Q_1$ is the unique radial positive $H^1$-solution of 
	\begin{align}\label{GS-equation}
		\Delta Q   - \phi_{|Q|^2} Q =  Q,\;\; x \in \R^4.
	\end{align}
	Such $Q$ is called the \emph{ground state}. The existence follows from the adaption of \cite{Weinstein} and for uniqueness we refer to \cite{Lieb} (in $d=3$), as well as \cite[Section 4]{KLR} in $ d=4$. \\ [3pt]
	By using Gagliardo-Nirenberg interpolation and Hardy-Littlewood-Sobolev's inequality, 
	\[
	\mathcal{E}(u) \geq \| \nabla u \|_{L^2_x} \bigg(1 - \frac{\| Q\|_{L^2_x}}{\| u\|_{L^2_x}}\bigg),\;\; u \in H^1(\R^4),
	\]
and hence the solutions of \eqref{H} are global if $ u_0 \in H^1(\R^4) $ with $ \| u_0 \|_{L^2_x} < \| Q \|_{L^2_x}$. This is sharp due to the  \emph{pseudo-conformal symmetry}, i.e. if $u(t,x) $ solves \eqref{H}, then so does 
	\begin{align}
		u(t,x) \mapsto \frac{1}{t^2} \overline{u}(\frac{1}{t}, \frac{x}{t})e^{i \frac{|x|^2}{4t}}.
	\end{align}
	In particular, applying the symmetry to $e^{it}Q(x)$ we obtain a \emph{minimal mass blow-up solution} 
	\begin{equation} \nonumber
		S(t,x) = \frac{1}{t^2} Q(\frac{x}{t}) e^{i( -  \frac{1}{t} +  \frac{|x|^2}{4t})},
    \end{equation}
    which satisfies $\| S(t) \|_{L^2_x} = \| Q\|_{L^2_x}$ and $\|\nabla S(t) \|_{L^2_x} \sim |t|^{-1},\; t \to 0^-$. For radial data, Krieger-Lenzmann-Rapha\"el \cite{KLR} proved the finite-codimension stability of $S(t)$. Besides the pseudo-conformal rate, formal and numerical calculations in \cite{YRZ-Hartree} reveal further log-log blow-up mechanisms as for the $L^2$-critical (local) NLS. 
    
    \;\\
    \emph{Literature}. Recently, in \cite{GSW}, the authors have constructed multisoliton  solutions of \eqref{H} with modulation trajectories following  non-trapped solutions (parabolic and hyperbolic trajectories) of an $m$-body problem derived from the Newton potential $|x|^{-2}$. Via the pseudo-conformal invariance, this leads to blow-up at any number of distinct points of multiplicity one, and collision blow-up at a single point of any multiplicity (see \cite[Section 1.3]{GSW}). \\[4pt]
    The method of proof applied in \cite{GSW} was pioneered for the 3D Hartree-NLS by Krieger-Martel-Rapha\"el \cite{KRM} for two-soliton solutions along asymptotically hyperbolic/parabolic orbits. Subsequently, this approach was extended by the second author \cite{Wu} to 3D multisoliton solutions following hyperbolic, parabolic and hyperbolic-parabolic motion.\\[4pt]
    In comparison to \eqref{H}, the $L^2$-critical NLS
	\begin{align} \label{mass-NLS}
		i \partial_t u + \Delta u  = - |u|^{\frac{4}{d}}u \;\; \text{in} \;\; \R^{1+d} 
	\end{align}
	is well studied concerning $H^1$-blow-up and solitary dynamics, see the overview in \cite[Section 1]{GSW} and \cite[Section 1]{MR} for instance.  Multisoliton solutions for \eqref{mass-NLS} have been proved to exist in \cite{Merle-k-bl}, \cite{MR}, \cite{PR}, \cite{Fan} and generally provide a description  of large data solutions, see in particular \cite{MM},  \cite{Koch-T}, \cite{RodSS}, \cite{CMM},     \cite{P} for the construction and stability results in NLS equations.\\[4pt]
    More relevant for our present work is the \emph{strong interaction regime}, see \cite{KRM}, \cite{Wu}, \cite{MR}, \cite{N} for which we also refer to strongly interacting kinks of $(1+1)$-scalar fields \cite{J1}, \cite{J2}. In particular in the latter work a (very different) 1-dimensional $m$-body approximation is used to derive the trajectories to leading order.
    \\[8pt]
    \emph{\textbf{In this article}}, we extend the multisoliton result \cite[Section 1.3]{GSW} and \emph{obtain solutions of \eqref{H} with modulation trajectories asymptotic to hyperbolic-parabolic orbits of  a suitable $m$-body law \eqref{m-body}} (see Theorem \ref{main-thm}).\\[4pt] Further, by applying the pseudo-conformal transformation, we obtain  \emph{finite-time collision blow-up at the pseudo-conformal rate with multiple soliton trajectories  colliding simultaneously at multiple distinct prescribed points} (see Corollary \ref{Cor blow up}).  
    
    \;\\
\emph{The $m$-body law}. Let $ m \in \Z_{\geq 2} $ and $ \alpha = (\alpha_1, \alpha_2, \dots, \alpha_m) \in \R^{ 4m} $ be the configuration of $ m $ bodies with center $ \alpha_j \in \R^4$. We consider (c.f.  \cite[Section 1.2]{GSW}) the following \emph{$m$-body law} 
	\begin{align}  
    \label{m-body}
		\begin{cases}
			\;\;\dot{\alpha}_j(t) =  2\beta_j(t),\; \;1 \leq  j \leq m,\;&\;\;\\[5pt]
			\;\;\displaystyle	\dot{\beta}_j(t) =  -  \| Q \|^2_{L^2_x} \sum_{ k \neq j }   \frac{\alpha_j - \alpha_k}{|\alpha_j - \alpha_k|^4}.&
		\end{cases}
	\end{align}
	The system \eqref{m-body} has a  first integral $H = K - U$, where
	\begin{align}\label{energy}
	K(\beta) =  2\sum_{j =1}^m |\beta_j(t)|^2,\;\;  U(\alpha) =  \| Q \|^2_{L^2_x} \sum_{j < k} \frac{ 1}{|\alpha_j - \alpha_k|^2},
	\end{align} 
    are the \emph{kinetic-} and \emph{potential energy} respectively.
    Following the  idea in \cite{KRM,Wu} and \cite{GSW}, the  trajectories  of solutions $(\alpha(t), \beta(t))$ to \eqref{m-body} will provide leading asymptotics of the modulation variables through applying \eqref{inv} to the ground state $Q$.\\[4pt]
	The \emph{center of mass} $M(t) := \sum_{j=1}^m  \alpha_j(t) \displaystyle$ evolves by free Galilean motion. Let us recall (see \cite{GSW}) the  set of \emph{central configurations} and the \emph{non-collision set}
	\begin{align*}
		&\mathcal{X} : = \big \{ x = (x_1, x_2, \dots, x_m) \in \R^{4m}\;|\; { \textstyle   \sum_{j =1}^m  x_j = 0 } \big \},\;\; \mathcal{Y} : = \mathcal{X} \setminus \Delta,\\
		& \Delta : = \big \{ x = (x_1, x_2, \dots, x_m) \in \mathcal{X}\;|\; \exists\; j \neq k: x_j = x_k \big \}.
	\end{align*}
  \;\\
Further we set  $\alpha_{j k} = \alpha_j - \alpha_k$ and the \emph{minimal distance} $a:= \min_{j < k} |\alpha_{jk}|$. For global solutions $\alpha(t)$ to \eqref{m-body} we have in particular (c.f. the argument in \cite{Marchal-Saari})
\[
    \max_{j < k}|\alpha_{j k}(t)| = O(t)\;\text{as}\;t \to +\infty,\;\; \liminf_{t \to +\infty} a(t) > 0.
\]
Moreover, for some $ v \in \mathcal{X}$ (the \emph{limit velocity}) we have $\alpha(t) = v t + O(t^{\f12})$ as $t \to +\infty$.
\\[8pt]
\emph{Expansive orbits.} If $a(t) \to \infty$ as $ t \to +\infty$, we say $\alpha(t)$ is an \emph{expansive orbit}. Let us recall from \cite[Section 1.2]{GSW} the following definition for expansive solutions.\\
\begin{enumerate} \setlength\itemsep{4pt}
    \item We say $(\alpha(t), \beta(t))$ is \underline{\emph{hyperbolic}}, if $v \in \mathcal{Y}$, i.e., $|\alpha_j(t) - \alpha_k(t)| \sim t$ as $ t \to +\infty$ for all $ j \neq k$.
    \item We say $(\alpha(t), \beta(t))$ is \underline{\emph{parabolic}}, if $v = 0$ and $|\alpha_j(t)- \alpha_k(t)| \sim t^{\f12} $ as $ t \to +\infty$ for all $ j \neq k$.
    \item We say $(\alpha(t), \beta(t))$ is \underline{\emph{hyperbolic-parabolic}}, if $v \in \Delta \setminus \{0\}$ and $|\alpha_j(t)- \alpha_k(t)|\sim t^{\f12}$ as $t \to +\infty$ for all $j \neq k$ such that $v_j= v_k$.
    \end{enumerate}
\vspace{8pt} 
\noindent
\emph{Clusters}. In the hyperbolic-parabolic case (3), we define a \emph{cluster partition} via the equivalence relation $ j \sim k$ if and only if $v_j = v_k$, i.e., $|\alpha_j(t) - \alpha_k(t)| \sim t^{\f12}$. Each equivalence class is called a \emph{cluster}. Now assuming there are $l$ \emph{clusters} $\{K_i\}_{i=1}^l$, then we may define the \emph{cluster energy} 
\[
    U_{K_i}(\alpha) = \| Q \|^2_{L^2_x} \sum_{\substack{j < k\\ j,k \in K_i}} \frac{ 1}{|\alpha_j - \alpha_k|^2}, \quad \tilde{U}(\alpha) = \sum_{i=1}^{l} U_{K_i}(\alpha),
\]
and thus $U(\alpha) = \tilde{U}(\alpha) + O(t^{-2})$ as $t \to + \infty$. 
\;\\[5pt]
The following proposition guaranties the existence of hyperbolic-parabolic solutions of \eqref{m-body} for any prescribed limiting velocities.

\begin{prop} \label{prop existence hyp-parab}
Let $v \in \Delta \setminus \{0\}$ and $\{K_i\}$ be the clusters. For each $i$, assume $(\alpha_j^0, \beta_j^0)$, $j \in K_i$ is a parabolic solution to the $|K_i|$-body problem. Then there exists a hyperbolic-parabolic solution $(\alpha, \beta)$ to the $m$-body problem \eqref{m-body} with 
\begin{equation}
    \alpha_j(t)= v_j t+ \alpha_j^0(t)+ o(t^{-1+}) \;\; \text{as} \;\; t \to +\infty.
\end{equation}
\end{prop}

\begin{proof}[Sketch of the proof]
Let $\tilde{\alpha}_j(t)= v_jt+ \alpha_j^0(t)$. Then 
\begin{equation}
    \ddot{\tilde{\alpha}}= \nabla \tilde{U}(\tilde{\alpha})= \nabla U(\tilde{\alpha})+ O(t^{-3}).
\end{equation}
For $x= \alpha- \tilde{\alpha}$, by the Taylor formula, we obtain 
\begin{equation}
    \ddot{x}= \nabla U(\alpha)- \nabla U(\tilde{\alpha})+ O(t^{-3})= \nabla^2 U(\tilde{\alpha})x+ O(t^{-3})
\end{equation}
with the bootstrap assumption $x(t)= o(t^{-1+})$. Further, there exists $A \in \R^{4m \times 4m}$ and $\epsilon>0$ such that $\nabla^2 U(\tilde{\alpha})= t^{-2}A+ O(t^{-2-\epsilon})$. Thus we have 
\begin{equation}
    \ddot{x}= \frac{Ax}{t^2}+ O(t^{-3}),
\end{equation}
for which we may use the argument in \cite[Lemma\;4.4]{GSW} to find $x$.
\end{proof}

We refer to \cite[Theorem~3]{Wu2} for a more general result with detailed proof.

\subsection{Statement of the results} Here we state our main results in this article, which is an extension of the construction in \cite[Section 1.3]{GSW} to the case of   hyperbolic-parabolic trajectories.

\begin{Th} \label{main-thm}
Let $(\alpha^\infty, \beta^\infty)$ be a hyperbolic-parabolic solution to \eqref{m-body} and $\lambda_j^\infty>0$ such that for $j \neq k$, $\lambda_j^\infty= \lambda_k^\infty$ whenever $|\alpha_j^\infty(t)- \alpha_k^\infty(t)| \sim t^\f12$ as $t \to +\infty$. Then there exist a solution $u \in C^0( [0, +\infty), H^1(\R^4)) $ to \eqref{H} and $\gamma^\infty(t)$ with 
\begin{equation}
	\bigg \| u(t,\cdot)- \sum_{j=1}^m Q_{\lambda_j^{\infty}} \big( \cdot-\alpha_j^\infty(t) \big) e^{i\gamma_j^\infty(t)+ i\beta_j^\infty(t) \cdot x} \bigg \|_{H^1}= o(t^{-\f12+}) \quad \text{as } t \to +\infty.
\end{equation}
\end{Th}

\begin{Rem}
Let $K_1, \cdots, K_l$ be the clusters for $(\alpha^\infty, \beta^\infty)$. Then for any $1 \le i \le l$, we consider 
\begin{equation}
    \sum_{j \in K_i} Q_{\lambda_j^{\infty}} \big( \cdot-\alpha_j^\infty(t) \big) e^{i\gamma_j^\infty(t)+ i\beta_j^\infty(t) \cdot x}
\end{equation}
to be a \emph{soliton cluster}. For each soliton cluster, there exists a Galilean transform such that the solitons within the cluster have vanishing velocities.
\end{Rem}
\vspace{4pt} \noindent
Recall $\Delta = \big \{ x = (x_1, x_2, \dots, x_m) \in \R^{4m} \;|\; \sum_{j=1}^m x_j=0 \;\;\text{and} \;\; \exists\; j \neq k: x_j = x_k \big \}$ is the collision set, and $\Delta \setminus \{0\}$ represents the set of non-parabolic collision configurations. 

\begin{Cor}[Finite-time collision blow-up at multiple points with multiplicity] \label{Cor blow up} Let $m \in \Z_{\geq 2}$ and $v \in \Delta \setminus \{0\}$. Then \eqref{H} has a solution $ u \in C^0((-\infty,0), H^1(\R^4))$ that blows up at $t=0$ such that
\begin{align*}
	\|\nabla u(t) \|_{L^2_x} \sim |t|^{-1} \;\;\;\text{and}\; \;\; |u(t)|^2 \rightharpoonup \sum_{j=1}^m \|Q\|_{L^2_x}^2 \delta_{v_j} \;\; \text{as} \;\; t \to 0^-.
\end{align*}
Moreover, for any $\lambda_1, \cdots, \lambda_m>0$ satisfying $\lambda_j= \lambda_k$ whenever $v_j= v_k$, and any hyperbolic-parabolic solution $(\tilde{\alpha}(t), \tilde{\beta}(t))$ of the $m$-body problem \eqref{m-body} satisfying
\begin{equation}
    \lim_{t \to +\infty} t^{-1} \tilde{\alpha}_j(t)= v_j, \quad \forall j=1, 2 \cdots m,
\end{equation} 
there is a blow-up solution $u$ as above and $\gamma_j \in C^0(- \infty, 0)$, $  \alpha_j(t) = |t| \tilde{\alpha}_j(|t|^{-1})$,  $  \beta_j(t) = -|t|^{-1} \tilde{\beta}_j(|t|^{-1})$, such that when writing 
\begin{equation} \label{eq expansion pseudo}
    u(t, x) =  \sum_{j=1}^m \frac{1}{t^2} Q_{\lambda_j}\Big( \frac{ x- \alpha_j(t) }{t}\Big) e^{i( \gamma_j(t) +  \beta_j(t) \cdot x + \frac{|x|^2}{4t})}+ \varepsilon(t,x),
\end{equation}
we have $\displaystyle \lim_{t \to 0^-}\| \varepsilon(t) \|_{L^2 \cap L^{\f83^-}} = 0$.
\end{Cor}
\;\\ [-10pt]
The proof of Corollary \ref{Cor blow up} follows from Theorem \ref{main-thm} and the application of the pseudo-conformal transformation. 

\begin{Rem}
The assumption $\sum_j x_j =0$ in the definition of $\Delta$ is without loss of generality by Galilean invariance of both \eqref{m-body} and \eqref{H}.    
\end{Rem}

\begin{Rem} Examples of hyperbolic-parabolic $m$-body trajectories are obtained from Proposition \ref{prop existence hyp-parab} together with \cite[Proposition 1.5]{GSW}. To be precise, for the hyperbolic-parabolic orbits $(\tilde{\alpha}, \tilde{\beta})$ we constructed, there exists $b_j \in \R^4$ such that
\[
    \tilde{\alpha}_j(t) = v_j t + b_j t^{\f12} + o(t^{-1+}), \;\; \tilde{\beta}_j(t)= \frac{v_j}{2}+ o(t^{-\f12+}) \;\;\;\text{as}\;\;t \to +\infty, \;\; \forall \;1 \le j \le m.
\]
Thus, we observe for $(\alpha, \beta)$ in Corollary \ref{Cor blow up}
\begin{equation}
    \alpha_j(t) = v_j + b_j |t|^{\f12} + O(|t|), \;\; \beta_j(t)=- \frac{v_j}{2|t|}+O(|t|^{-\f12}) \;\; \text{as} \;\;\; t \to 0^-, \;\; \forall \; 1 \le j \le m.
\end{equation}
\end{Rem}

\begin{Rem} (i)\;The solutions in Theorem \ref{main-thm} and Corollary \ref{Cor blow up} are \emph{strongly interacting} in the sense of \cite{MR}, i.e. the trajectories of the solitons are perturbed to leading order by the presence of the other solitons. An intricate  question for finite-time blow-up of the $L^2$-critical NLS is whether concentration scenarios 
\begin{align} \label{conc}
    |u(t)|^2 \rightharpoonup \sum_{i =1}^l m_i\delta_{x_i} + |u^*|^2 \;\; \text{as} \;\;t \to T^-,
\end{align}
appear for large masses (c.f. \cite{Mer-Raph} in the weakly interacting regime) and which rates are possible. Here $m_i > 0$, $x_1, \cdots, x_l$ are distinct and $u^* \in L^2$. Corollary \ref{Cor blow up} shows that for \eqref{H}, the asymptotics \eqref{conc} with $u^* = 0$, $l \in \Z_{\geq 2}$ and $m_i = \|Q\|_{L^2}^2 k_i $, $k_i \in \Z_+$ can be realized \emph{at the pseudo-conformal rate} via strong interactions.\\[4pt]
(ii)\;To the authors' knowledge, collision blow-up for the mass-critical NLS \eqref{mass-NLS} is only known by the remarkable result in  \cite{MR}, where the blow-up is at a single point above the pseudo-conformal rate. 
\end{Rem}

\vspace{4pt} \noindent
\emph{Outline}. The article follows the strategy proposed in \cite{KRM}, \cite{Wu} and \cite{GSW}. \\[4pt] 
In Section \ref{sec:gs-properties}, we quote the non-degeneracy and the (finite co-dimension) coercivity properties of the linearized operators from \cite{GSW}. \\[4pt]
In Section \ref{sec:approx}, we give details for the adapted construction of approximate solutions used in \cite{Wu}, \cite{GSW} (see Remark \ref{Rem-after admissibility}), which ensures a small interaction error. \\ [4pt]
This is followed by an analysis in Section \ref{sec:traje} for the choice of the modulation parameters to satisfy the modulation ODE \eqref{eq parameters}.\\[4pt]
In Section \ref{sec:bootstrap}, we reduce the proof to a bootstrap argument by choosing a modulation path with suitable orthogonality conditions. \\ [4pt]
In the last Section \ref{sec:mod}, we then control the modulation part of the bootstrap estimate by the analysis in Section \ref{sec:approx} and Section \ref{sec:traje}, and lastly use coercivity and the orthogonality conditions to control the error in the energy space via the bootstrap estimate.

\section{The ground state linearized operators}\label{sec:gs-properties}

We linearize \eqref{H} at the ground state $Q$ via writing $u(t) = e^{it}(Q + \varepsilon(t))$ and separating $ \varepsilon = \varepsilon_1 + i \varepsilon_2$. This shows the leading operator has the form
\begin{align} \label{lin-op-matrix}
 \begin{pmatrix}
    0 &  -L_-\\
    L_+ & 0
\end{pmatrix}  =  \begin{pmatrix}
    0 &  \Delta - 1\\
    -\Delta + 1  & 0
\end{pmatrix}  + 
\begin{pmatrix}
    0 & - \phi_{Q^2}\\
    \phi_{Q^2} + 2 Q \phi_{(\cdot\;  Q)} & 0
\end{pmatrix}
\end{align}
when acting on $\begin{pmatrix}
    \varepsilon_1\\
    \varepsilon_2
\end{pmatrix}$ and hence we identify 
\begin{align}\label{lin-op}
L_{+} = -\Delta + 1 + \phi_{Q^2} + 2 Q \phi_{(\cdot\;  Q)},\;\;L_{-} = -\Delta + 1 + \phi_{Q^2}. 
\end{align}
In particular $L_{\pm}$ are posed on  $L^2(\R^4)$ with domain $ D = H^2(\R^4)\subset L^2(\R^4)$ and we recall 
\[
\phi_{Q^2} = -  |x|^{-2}\ast Q^2,\;\;\phi_{(\cdot \; Q)} = -   |x|^{-2} \ast (\;\cdot \; Q),\;\; 
\]
thus $L_+$ is a non-local operator. 

Let us now recall important non-degeneracy  and inversion properties of $L_{\pm}$ proved in \cite[Section 2]{GSW}, see also \cite{KLR}, \cite{LenzmannNondegen}. We start with the generalized root space.
\begin{lem}[\cite{GSW}] \label{lem straight}\phantom{a}
\begin{itemize}
    \item[(i)] $L_{\pm}$ are self-adjoint, $ \emph{spec}_{\emph{c}}(L_\pm) = [1, +\infty)$, $ 0 \in \emph{spec}(L_\pm)$ and $L_- \geq 0$.
    
    \item[(ii)] The  following are elements of the generalized root space of the operator \eqref{lin-op-matrix} 
    \begin{align*}
    \begin{pmatrix}
        0\\
        Q
    \end{pmatrix}, \;\begin{pmatrix}
        \Lambda Q\\
        0
    \end{pmatrix}, \; \begin{pmatrix}
        \partial_{x_j} Q\\
        0
    \end{pmatrix},\; \begin{pmatrix}
        0\\
        |x|^2Q
    \end{pmatrix},\; \begin{pmatrix}
        \rho(x)\\
        0
    \end{pmatrix},\;  \begin{pmatrix}
        0\\
        x_jQ
    \end{pmatrix},\;\; j = 1,2,3,4,
    \end{align*}
    where $\Lambda = 2 + x \cdot \nabla$ and 
    \begin{align}
    &L_-Q = L_+(\partial_{x_j}Q) = 0,\; L_+(\Lambda Q) = -2 Q,\; L_-(|x|^2 Q) = -4 \Lambda Q,\\
    & L_-(x_j Q) = -2\partial_{x_j} Q,\; L_+\rho = - |x|^2 Q.
    \end{align}
    \item[(iii)] We have 
    \begin{align*}
    &(Q,\rho) =   \f12 (\Lambda Q,|x|^2 Q) =   -\f12\| x Q\|_{L^2}^2, \;(Q, \Lambda Q) = 0,\\
    &(x Q, \nabla Q ) = - 2 \|Q\|^2_{L^2}. 
    \end{align*}
\end{itemize}
\end{lem}
The following two Propositions will be essential in the next Section \ref{sec:approx} and Section \ref{sec:bootstrap}.
\begin{prop}[\cite{GSW}] \label{non-deg-coer-Inv} \

\begin{enumerate}
    \item We have non-degeneracy of the kernel of $L_{\pm}$, i.e. if $L_{+} u = 0$, then $ u = a \cdot \nabla Q$ for some $ a \in \R^4$. Likewise if $L_{-} u = 0$, then $ u = b Q$ for some $ b \in \R$.
    \item  For all real-valued functions $ v \in H^1$ we have the coercivity 
\begin{equation} \begin{aligned} \label{coer-L+-}
    (L_+v, v) &\ge c \| v \|_{H^1}^2- C(v,Q)^2- C(v,xQ)^2- C(v,|x|^2 Q)^2, \\ 
    (L_-v, v) &\ge c \| v \|_{H^1}^2- C(v,\rho)^2,
\end{aligned} \end{equation}
where  $c, C > 0$ are positive constants (independent of $v$).
\end{enumerate}
\end{prop}
\begin{prop}[\cite{GSW}] \label{prop inversion of L_pm}
Let $f$ be a real-valued admissible function of degree $n \in \N$. Then we have the following
\begin{enumerate}
    \item If $\langle f , \nabla Q \rangle = 0$, there exists a real-valued  admissible solution to $ L_+u = f$ of degree $n$.
    \item If $\langle f ,Q \rangle = 0$,  there exists a real-valued admissible solution  to $ L_-u = f$ of degree $n$. Further if $ f$ is radial, then $u$ can be chosen radial.
\end{enumerate}
\end{prop}

\section{Construction of approximate solutions}\label{sec:approx}
\noindent
We let $\alpha= (\alpha_1, \cdots, \alpha_m)$ and similarly for $(\beta$, $\lambda$, $\mu$, $\delta$, $\gamma$), which are (possibly) time dependent parameters . We also denote 
\begin{equation} \label{eq notation} \begin{split}
    &P= (\alpha, \beta, \lambda, \mu, \delta), \quad g=(P,\gamma), \quad g_j = (P_j, \gamma_j) =  (\alpha_j, \beta_j, \lambda_j, \mu_j, \delta_j, \gamma_j),\\[3pt]
    &\alpha_{jk}= \alpha_j- \alpha_k, \quad \beta_{jk}= \beta_j- \beta_k, \quad a= \min_{j \neq k} |\alpha_{jk}|.
\end{split} \end{equation}
For a function $v : \R \times \R^4 \to \C$ we then modulate via the path $g_j =(P_j, \gamma_j)$
\begin{equation} \label{eq action}
    g_j v(t,x)  := \lambda_j^2 v \big(t, \lambda_j (x-\alpha_j) \big) e^{i\gamma_j+ i\beta_j \cdot x+ i \mu_j |x|^2}.
\end{equation}
In particular, we set $u_j(t,x) = g_j v_j(t,x)$ and assume the $v_j$'s are functions of $\alpha_j$, $\beta_j$, $\lambda_j$, $\gamma_j$, $\mu_j$ and the spatial variable. If we set 
\begin{equation}
    u(t,x)= \sum_{j=1}^m u_j(t,x)= \sum_{j=1}^m g_jv_j(t,x)
\end{equation}
then it is, after some calculations, straightforward to observe the expression
\begin{equation}
    i\partial_t u+ \Delta u- \phi_{|u|^2}u= \sum_{j=1}^m E_j(t, y_j) e^{i\gamma_j+ i\beta_j \cdot x+ i \mu_j |x|^2}- \sum_{k \neq j} \phi_{\Re (u_k \overline{u_j})} u,
\end{equation}
where $y_j= \lambda_j (x-\alpha_j)$ and 
\begin{equation} \begin{aligned}
    E_j(t, y_j) &=  \lambda_j^4 \big( \Delta v_j- v_j- \delta_j |y_j|^2 v_j \big) \\
    &\quad - \big( \dot{\mu}_j+ 4\mu_j^2- \lambda_j^4 \delta_j \big) |y_j|^2 v_j- \lambda_j \big( \dot{\beta}_j+ 4\mu_j \beta_j+ (\dot{\mu}_j+ 4\mu_j^2) \alpha_j \big) y_j v_j \\
    &\quad - \lambda_j^2 \big( \dot{\gamma}_j+ (\dot{\beta}_j+ 4\mu_j \beta_j) \cdot \alpha_j+ (\dot{\mu}_j+ 4\mu_j^2) |\alpha_j|^2+ |\beta_j|^2- \lambda_j^2 \big) v_j \\
    &\quad +i\lambda_j \big( \dot{\lambda}_j+ 4\lambda_j \mu_j \big) \Lambda v_j- i\lambda_j^3 \big( \dot{\alpha}_j- 2\beta_j- 4\mu_j \alpha_j \big) \nabla v_j \\
    &\quad+ i \lambda_j^2 \sum_{k=1}^m \Big( \frac{\partial v_j}{\partial \alpha_k} \dot{\alpha}_k+ \frac{\partial v_j}{\partial \beta_k} \dot{\beta}_k+ \frac{\partial v_j}{\partial \lambda_k} \dot{\lambda}_k+ \frac{\partial v_j}{\partial \mu_k} \dot{\mu}_k+ \frac{\partial v_j}{\partial \delta_k} \dot{\delta}_k \Big) \\
    &\quad - \lambda_j^4 \bigg[ \phi_{|v_j|^2}+ \sum_{k \neq j} \Big( \frac{\lambda_k}{\lambda_j} \Big)^2 \phi_{|v_k|^2} \Big( \frac{\lambda_k y_j}{\lambda_j}+ \lambda_k \alpha_{jk} \Big) \bigg] v_j. 
\end{aligned} \end{equation}
\;\\
Since we have
\begin{equation}
    \left( \frac{\lambda_k}{\lambda_j} \right)^2 \phi_{|v_k|^2} \Big( P(t), \frac{\lambda_k y_j}{\lambda_j}+ \lambda_k \alpha_{jk} \Big)= -\frac{1}{\lambda_j^2} \int_{\mathbb{R}^4} \frac{|v_k(P(t),\xi)|^2}{|\alpha_{jk}+ \lambda_j^{-1} y_j- \lambda_k^{-1} \xi|^2} \d \xi,
\end{equation}
we now consider the Taylor expansion
\begin{equation}
    \frac{1}{|\alpha-\zeta|^2}= \sum_{n=1}^N F_n(\alpha,\zeta)+ O \left( \frac{|\zeta|^N}{|\alpha|^{N+2}} \right) \quad \text{as } \zeta \to 0,
\end{equation}
where $F_n(\alpha,\zeta)$ is homogeneous of degree $-n-1$ in $\alpha$ and of degree $n-1$ in $\zeta$. 

Similar to \cite{KRM}, \cite{Wu} and as in \cite[Section 3]{GSW}, let us define the approximation to be
\begin{equation} \begin{aligned}
    \phi_{|v_k|^2}^{(N)}(t,y_j) &:= \sum_{n=1}^N \psi_{|v_k|^2}^{(n)}(t,y_j) := \sum_{n=1}^N -\frac{1}{\lambda_j^2} \int_{\mathbb{R}^4} |v_k(t,\xi)|^2 F_n \big( \alpha_{jk}, \lambda_k^{-1} \xi- \lambda_j^{-1} y_j\big) \d \xi,
\end{aligned} \end{equation}
and then make the ansatz
\[
v_j(t,y_j):= V_j^{(N)}(P(t),y_j):= Q(y_j)+ \delta_j \rho(y_j)+ W_j^{(N)}(P(t), y_j),
\]
where the function $\rho(y_j)$ is as in Lemma \ref{lem straight}.  We define
\begin{equation} \label{eq approximate solution}
    R_g^{(N)}(t,x):= \sum_{j=1}^m R_{j,g}^{(N)}(t,x):= \sum_{j=1}^m g_j V_j^{(N)}(P(t),x). 
\end{equation}
as the approximate solution. Omitting the subscript $g$ of $R^{(N)}$ for convenience, we write
\begin{equation} \label{eq equation of approximate solution} \begin{aligned}
    &i\partial_t R^{(N)}+ \Delta R^{(N)}- \phi_{|R^{(N)}|^2} R^{(N)} \\
    &=\sum_{j=1}^m E_j^{(N)}(t,y_j) e^{i\gamma_j+ i\beta_j \cdot x +i\mu_j|x|^2}- \sum_{k \neq j} \phi_{\mathrm{Re} (R_k^{(N)} \overline{R_j^{(N)}})} R^{(N)} \\
    &+ \sum_{j=1}^m \lambda_j^4 \sum_{k \neq j} \bigg[ \phi_{\left| V_k^{(N)} \right|^2}^{(N)}- \Big( \frac{\lambda_k}{\lambda_j} \Big)^2 \phi_{ \left| V_k^{(N)} \right|^2} \Big( \frac{\lambda_k y_j}{\lambda_j}+ \lambda_k \alpha_{jk} \Big) \bigg] V_j^{(N)} e^{i\gamma_j+ i\beta_j \cdot x+i\mu_j|x|^2},
\end{aligned} \end{equation} 
where now $E_j^{(N)}= \tilde{E}_j^{(N)}+ S_j^{(N)}$ with 
\begin{equation} \label{eq definition of E_j tilde} \begin{aligned} 
    &\tilde{E}_j^{(N)}(t,y_j) \\
    = & \ \lambda_j^4 \Big( \Delta V_j^{(N)}- V_j^{(N)}- \phi_{\left| V_j^{(N)} \right|^2} V_j^{(N)}- \delta_j |y_j|^2 V_j^{(N)} \Big)- \lambda_j^4 \sum_{k \neq j} \phi_{\left| V_k^{(N)} \right|^2}^{(N)} V_j^{(N)} \\
    &+ i\lambda_j^2 D_j^{(N)} \rho- \lambda_j B_j^{(N)} \cdot y_j V_j^{(N)}+ i\lambda_j M_j^{(N)} \Lambda V_j^{(N)} \\
    &+ i\lambda_j^2 \sum_{k=1}^m \bigg( \frac{\partial W_j^{(N)}}{\partial \alpha_k} \cdot \big( 2\beta_k+ 4\mu_k \alpha_k \big)+ \frac{\partial W_j^{(N)}}{\partial \beta_k} \cdot \big( B_k^{(N)}- 4\mu_k \beta_k- \lambda_k^4 \delta_k \alpha_k \big) \\
    &\qquad \qquad \quad \ + \frac{\partial W_j^{(N)}}{\partial \lambda_k} \big( M_k^{(N)}- 4\mu_k \lambda_k \big)+ \frac{\partial W_j^{(N)}}{\partial \mu_k} \big( \lambda_k^4 \delta_k- 4\mu_k^2 \big)+ \frac{\partial W_j^{(N)}}{\partial \delta_k} D_k^{(N)} \bigg) 
\end{aligned} \end{equation} 
and
\begin{equation} \label{eq definition of S_j^N} \begin{aligned}
    &S_j^{(N)}(t,x) \\
    = &- i\lambda_j^3 \big( \dot{\alpha}_j- 2\beta_j- 4\mu_j \alpha_j \big) \nabla V_j^{(N)}+ i\lambda_j \big( \dot{\lambda}_j+ 4\lambda_j \mu_j- M_j^{(N)} \big) \Lambda V_j^{(N)}+ i\lambda_j^2 \big( \dot{\delta}_j- D_j^{(N)} \big) \rho \\
    &- \big( \dot{\mu}_j+ 4\mu_j^2- \lambda_j^4 \delta_j \big) |y_j|^2 V_j^{(N)}- \lambda_j \big( \dot{\beta}_j+ 4\mu_j \beta_j+ (\dot{\mu}_j+ 4\mu_j^2) \alpha_j- B_j^{(N)} \big) \cdot y_j V_j^{(N)} \\
    &- \lambda_j^2 \big( \dot{\gamma}_j+ (\dot{\beta}_j+ 4\mu_j \beta_j) \cdot \alpha_j+ (\dot{\mu}_j+ 4\mu_j^2) |\alpha_j|^2+ |\beta_j|^2- \lambda_j^2 \big)  V_j^{(N)} \\
    &+ i\lambda_j^2 \sum_{k=1}^m \Bigg[ \frac{\partial W_j^{(N)}}{\partial \alpha_k} \cdot \left( \dot{\alpha}_k- 2\beta_k- 4\mu_k \alpha_k \right)+ \frac{\partial W_j^{(N)}}{\partial \beta_k} \cdot \left( \dot{\beta}_k+ 4\mu_k \beta_k+ \lambda_k^4 \delta_k \alpha_k- B_k^{(N)} \right) \\
    &\quad \ + \frac{\partial W_j^{(N)}}{\partial \lambda_k} \big( \dot{\lambda}_k+ 4\mu_k \lambda_k- M_k^{(N)} \big)+ \frac{\partial W_j^{(N)}}{\partial \mu_k} \big( \dot{\mu}_k+ 4\mu_k^2- \lambda_k^4 \delta_k \big)+ \frac{\partial W_j^{(N)}}{\partial \delta_k} \big( \dot{\delta}_k- D_k^{(N)} \big) \Bigg].
\end{aligned} \end{equation}
Here the terms $ \tilde{E}_j^{(N)}$ are controlled by the nonlinear interactions and $S_j^{(N)} $  contains the modulation error. In particular, the approximation will be defined by comparing degrees of (parameter) homogeneity in  $ \tilde{E}_j^{(N)}$, where  $M_j^{(N)}$ are free to choose and $B_j^{(N)}, D_j^{(N)}$ are determined by orthogonality. The following notion of admissible functions will be useful.

\begin{defn}[\textbf{Admissible functions}] \label{Def-adm} \ \par
Recalling \eqref{eq notation}, let $\Omega$ denote the space of non-collision positions:
\begin{equation} \begin{aligned}
    \Omega:= \Big\{ P= (\alpha, \beta, \lambda, \mu, \delta) \in \mathbb{R}^{4m} \times \mathbb{R}^{4m} \times \mathbb{R}_+^m \times \mathbb{R}^m \times \mathbb{R}^m \ \big| \ \alpha_j \neq \alpha_k,\ \forall j \neq k \Big\}.
\end{aligned} \end{equation}

(1) Let $n \in \mathbb{N}$. Define $S_n$ to be the set of functions $\sigma: \Omega \to \mathbb{R}$ that is a finite sum of 
\begin{equation} \label{eq term in S_n}
    c\prod_{j \neq k} |\alpha_j- \alpha_k|^{-q_{jk}} (\alpha_j-\alpha_k)^{p_{jk}} \prod_{j=1}^m (\mu_j \alpha_j)^{s_j} (\delta_j \alpha_j)^{t_j} \beta_j^{k_j} \lambda_j^{l_j} \mu_j^{m_j} \delta_j^{n_j},
\end{equation}
where $c \in \mathbb{R}$, $q_{jk} \in \mathbb{N}$, $p_{jk} \in \N^4$, $|p_{jk}| \le q_{jk}$, $s_j, t_j, k_j \in \mathbb{N}^4$, $l_j \in \mathbb{Z}$, $m_j, n_j \in \mathbb{N}$, and
\begin{equation} \begin{gathered}
    \sum_{j \neq k} (q_{jk}- |p_{jk}|)+ \sum_{j=1}^m \big( 5m_j+ 7n_j+ 3|s_j|+ 5|t_j| \big)= n.
\end{gathered} \end{equation}

(2) We say a function $u: \Omega \times \mathbb{R}^4 \to \mathbb{C}$ is \textbf{admissible} if $u$ is a finite sum of 
\begin{equation}
    z\sigma(\alpha, \beta, \lambda, \mu, \delta) \tau(x),
\end{equation}
where $z \in \mathbb{C}$, $\sigma \in S_n$ for some $n \in \mathbb{N}$ and $\tau \in C^\infty$ satisfies
\begin{equation}
    \big| \nabla^k \tau(x) \big| \le e^{-c_k|x|}, \qquad \forall k \ge 0,\ x \in \mathbb{R}^4.
\end{equation}

If $n$ is the same for all addends, then we say $u$ is admissible of degree $n$. Otherwise, taking $n$ as the minimal one among all addends, we say $u$ is admissible of degree $\ge n$.
\end{defn}

\begin{Rem}\label{Rem-after admissibility} We stress that the above Definition \ref{Def-adm} part (1) is different compared to \cite[Section 3]{GSW}, and is crucial for us to extend the results there to hyperbolic-parabolic trajectories. The condition here is stronger. This is to compensate for the fact that $\alpha_j$ may grow faster than $\alpha_j-\alpha_k$ in the hyperbolic-parabolic case, \emph{as $|\alpha_j- \alpha_k| \sim t^{1/2}$ when $j,k$ are in the same cluster, while $|\alpha_j| \sim t$}. In other words, if we use the definition in \cite[Section 3]{GSW}, then admissible functions may not decay in time.
\end{Rem}

The following lemma states the required decay of admissible functions.

\begin{lem} \label{lem properties of admissible functions-2}
Let $n \in \mathbb{N}$, $u$ be admissible of degree $\ge n$ and $P$ satisfy 
\begin{equation} \label{eq boundedness of g}
    |\alpha| \lesssim a^2, \ |\beta| \lesssim 1, \ \lambda \sim 1, \ |\mu| \lesssim a^{-5}, \ |\delta| \lesssim a^{-7},
\end{equation} 
then there are constants $c_k > 0$ with 
\[
|\nabla^k u(P, x) | \lesssim a^{-n} e^{-c_k|x|},\;\; k \geq 0,\;x \in \R^4.
\]
\end{lem}

The following proposition gives the construction of approximate solutions. 

\begin{prop} \label{prop construction of approximate bubbles}
Given $m_j^{(n)} \in S_{n+1}$. For $n \ge 1$ and $1 \le j \le m$, there exist $d_j^{(n)}, b_j^{(n)} \in S_{n+1}$ and $T_j^{(n)}$ that is admissible of degree $n+1$ such that: for any $N \ge 1$, if we set
\begin{gather*}
    M_j^{(N)}(P)= \sum_{n=1}^N m_j^{(n)}(P), \;\;W_j^{(N)}(P,y_j)= \sum_{n=1}^N T_j^{(n)}(P,y_j),  \\
    D_j^{(N)}(P)= \sum_{n=1}^N d_j^{(n)}(P), \quad B_j^{(N)}(P)= \sum_{n=1}^N b_j^{(n)}(P),
\end{gather*}
then $\tilde{E}_j^{(N)}$ defined by \eqref{eq definition of E_j tilde} is admissible of degree $\ge N+2$.
\end{prop}

\begin{proof} The proof follows the identical procedure as in \cite[Proposition\;3.5]{GSW} with the modified Definition \ref{Def-adm} for admissible functions. 

Let us proceed by induction over $N$. Writing $T_j^{(n)}= X_j^{(n)}+ i Y_j^{(n)}$, we have
\begin{equation}  \begin{aligned}
    \tilde{E}_j^{(N+1)}- \tilde{E}_j^{(N)}= &- \lambda_j^4 \big( L_+ X_j^{(N+1)}+ iL_- Y_j^{(N+1)} \big)- \lambda_j^4 \sum_{k \neq j} \psi_{Q^2,k}^{(N+1)} Q \\
    &+ i\lambda_j^2 d_j^{(N+1)} \rho- \lambda_j b_j^{(N+1)} y_j Q+ i\lambda_j m_j^{(N+1)} \Lambda Q+ error,
\end{aligned} \end{equation}
where $error$ is (by induction) admissible of degree $\ge N+3$. For this, the key observation is that the terms in \eqref{eq definition of E_j tilde} do not depend solely on $\alpha_j$ but rather on the terms
\[
\frac{(\alpha_j - \alpha_k)^{p_{jk}}}{|\alpha_j - \alpha_k|^{q_{jk}}}\;(k \neq j, \; |p_{jk}| \le q_{jk}),\;\;\mu_j \alpha_j,\;\;\delta_j \alpha_j.\;\;
\]
which justifies the expression \eqref{eq term in S_n} given in Definition \ref{Def-adm}. Thus, in order to continue, it suffices to solve
\begin{equation} \label{this-1} 
    \left\{ \begin{aligned}
        &L_+ X_j^{(N+1)}= -\lambda_j^{-3} b_j^{(N+1)} \cdot y_j Q- \sum_{k \neq j} \psi_{Q^2,k}^{(N+1)} Q+ \lambda_j^{-4} \mathrm{Re} \ \hat{E}_j^{(N)}, \\
        &L_- Y_j^{(N+1)}= \lambda_j^{-2} d_j^{(N+1)} \rho+ \lambda_j^{-3} m_j^{(N+1)} \Lambda Q+ \lambda_j^{-4} \mathrm{Im} \ \hat{E}_j^{(N)},
    \end{aligned} \right.
\end{equation}
where $\hat{E}_j^{(N)}$ are the terms in $\tilde{E}_j^{(N)}$ of degree $N+2$. The existence of $X_j^{(N+1)}$ and $Y_j^{(N+1)}$ hence follows from Proposition \ref{prop inversion of L_pm} if
\begin{equation} \label{this-2}
    \left\{ \begin{aligned}
        & \Big( -\lambda_j^{-3} b_j^{(N+1)} \cdot y_j Q- \sum_{k \neq j} \psi_{Q^2,k}^{(N+1)} Q+ \lambda_j^{-4} \mathrm{Re} \ \hat{E}_j^{(N)}, \nabla Q \Big)=0, \\
        &\Big( \lambda_j^{-2} d_j^{(N+1)} \rho+ \lambda_j^{-3} m_j^{(N+1)} \Lambda Q+ \lambda_j^{-4} \mathrm{Im} \ \hat{E}_j^{(N)}, Q \Big)=0.
    \end{aligned} \right.
\end{equation}
Such $b_j^{(N+1)}$ and $d_j^{(N+1)}$ now always exist because $(y_j Q, \nabla Q)$ is invertible and $(\rho, Q) \neq 0$. 
\end{proof}

We conclude the section by a suitable estimate for the interaction error, which verifies the accuracy of approximate solutions we constructed above.
\begin{prop} \label{prop accuracy of approximate solution}
Let $V_j^{(N)}$ be as in Proposition \ref{prop construction of approximate bubbles}. For $R_g^{(N)}$ defined by \eqref{eq approximate solution}, let 
\begin{equation} \label{eq definition of Psi}
    \Psi^{(N)}= i\partial_t R_g^{(N)}+ \Delta R_g^{(N)}- \phi_{|R_g^{(N)}|^2} R_g^{(N)}- \sum_{j=1}^m S_j^{(N)} e^{i\gamma_j+ i\beta_j \cdot x+ i\mu|x|^2}.     
\end{equation}
If we assume \eqref{eq boundedness of g}, then there exist constants $ c , C > 0$  such that
\begin{equation} \label{eq estimate of Psi}
    |\Psi^{(N)}(t,x)| \le \frac{C}{a^{N+2}(t)}  \max \limits_{j = 1, \dots,m} e^{-c|x-\alpha_j(t)|},\;\; x \in \R^4.
\end{equation} 
\end{prop}

The proof is as given in \cite{GSW} and we refer to \cite[Section 3]{GSW} for details.

\section{Trajectories}\label{sec:traje}

We now want to solve the following modulation equations, which by definition is the condition under which $S_j^{(N)} =0$ is obtained.
\;\\
\begin{equation} \label{eq parameters}
    \left\{ \begin{aligned}
        &\dot{\alpha}_j- 2\beta_j- 4\mu_j \alpha_j=0, \\
        &\dot{\beta}_j+ 4\mu_j \beta_j+ \lambda_j^4 \delta_j \alpha_j- B_j^{(N)}=0, \\
        &\dot{\lambda}_j+ 4\lambda_j \mu_j- M_j^{(N)}=0, \\
        &\dot{\mu}_j+ 4\mu_j^2- \lambda_j^4 \delta_j=0, \\
        &\dot{\delta}_j- D_j^{(N)}=0.
    \end{aligned} \right.
\end{equation}

Let us start by setting $m_j^{(n)}=0$ for $n \neq 2$ and
\begin{equation}
    m_j^{(2)}= -\frac{2\| Q \|_{L^2}^2}{\lambda_j} \sum_{k \neq j} \frac{\alpha_{jk} \cdot \beta_{jk}}{|\alpha_{jk}|^4}.
    \end{equation}
We will show we can take
\begin{equation} \label{eq first few terms b}
    b_j^{(1)}= b_j^{(3)}= b_j^{(4)}=0, \quad b_j^{(2)}= -\| Q \|_{L^2}^2 \sum_{k \neq j} \frac{\alpha_{jk}}{|\alpha_{jk}|^4} 
\end{equation}
and 
\begin{equation} \label{eq first few terms d}
    d_j^{(1)}= d_j^{(2)}= d_j^{(3)}= d_j^{(4)}= d_j^{(5)}= d_j^{(6)}= 0. 
\end{equation}
We will also give an explicit expression of $d_j^{(7)}$. For convenience, we let 
\begin{equation} \label{eq f}
    f= \frac{1}{2} \sum_{k \neq j} \psi_{Q^2,k}^{(1)}= -\frac{\| Q \|_{L^2}^2}{2\lambda_j^2} \Big( \sum_{k \neq j} \frac{1}{|\alpha_{jk}|^2} \Big), \quad h= -\lambda_j \sum_{k=1}^m \frac{\partial f}{\partial \alpha_k} \cdot 2\beta_k= m_j^{(2)}. 
\end{equation}
Note that $f$ only depends on $\alpha$ and $\lambda_j$.  Recall from the proof of Proposition \ref{prop construction of approximate bubbles} (c.f. \cite[Proposition\;3.5]{GSW}) that we have
\begin{equation} 
    \left\{ \begin{aligned}
        &L_+ \re\ T_j^{(N+1)}= -\lambda_j^{-3} b_j^{(N+1)} \cdot y_j Q- \sum_{k \neq j} \psi_{Q^2,k}^{(N+1)} Q+ \lambda_j^{-4} \mathrm{Re} \ \hat{E}_j^{(N)}, \\
        &L_- \imm \ T_j^{(N+1)}= \lambda_j^{-2} d_j^{(N+1)} \rho+ \lambda_j^{-3} m_j^{(N+1)} \Lambda Q+ \lambda_j^{-4} \mathrm{Im} \ \hat{E}_j^{(N)},
    \end{aligned} \right.
\end{equation}
and
\begin{equation} 
    \left\{ \begin{aligned}
        & \Big( -\lambda_j^{-3} b_j^{(N+1)} \cdot y_j Q- \sum_{k \neq j} \psi_{Q^2,k}^{(N+1)} Q+ \lambda_j^{-4} \mathrm{Re} \ \hat{E}_j^{(N)}, \nabla Q \Big)=0, \\
        &\Big( \lambda_j^{-2} d_j^{(N+1)} \rho+ \lambda_j^{-3} m_j^{(N+1)} \Lambda Q+ \lambda_j^{-4} \mathrm{Im} \ \hat{E}_j^{(N)}, Q \Big)=0.
    \end{aligned} \right.
\end{equation}

Since $\hat{E}_j^{(0)}=0$ and $\big( \sum \limits_{k \neq j} \psi_{Q^2,k}^{(1)} Q, \nabla Q \big)=0$, we take
\begin{equation}
    b_j^{(1)}= d_j^{(1)}=0, \quad T_j^{(1)}= f \Lambda Q.
\end{equation}

Then we have 
\begin{equation}
    \hat{E}_j^{(1)}= i\lambda_j^2 \sum_{k=1}^m \frac{\partial T_j^{(1)}}{\partial \alpha_k} \cdot 2\beta_k= -i\lambda_j h \Lambda Q.
\end{equation}
Since $\re\ \hat{E}_j^{(1)}=0$, $\imm \big( E_j^{(1)}, Q \big)=0$ and by the choice of $m_j^{(2)}$, we may take
\begin{equation}
    b_j^{(2)}= -\| Q \|_{L^2}^2 \sum_{k \neq j} \frac{\alpha_{jk}}{|\alpha_{jk}|^4}, \quad d_j^{(2)}=0, \quad T_j^{(2)}=0. 
\end{equation}

Then we have 
\begin{align}
    \hat{E}_j^{(2)} &= -\lambda_j^4 \big( 2\phi_{QT_j^{(1)}} T_j^{(1)}+ \phi_{|T_j^{(1)}|^2} Q \big)- \lambda_j^4 \sum_{k \neq j} \Big( \psi_{Q^2,k}^{(1)} T_j^{(1)}+ 2\psi_{Q T_k^{(1)}}^{(1)} Q \Big) \\
    &= -\lambda_j^4 \big( 2\phi_{QT_j^{(1)}} T_j^{(1)}+ \phi_{|T_j^{(1)}|^2} Q \big)- \lambda_j^4 \sum_{k \neq j} \psi_{Q^2,k}^{(1)} T_j^{(1)} \\
    &= -\lambda_j^4 f^2 \Big( 2\phi_{Q \Lambda Q} \Lambda Q+ \phi_{(\Lambda Q)^2} Q+ 2\Lambda Q \Big).
\end{align}
By parity, we take  
\begin{equation}
    b_j^{(3)}= d_j^{(3)}=0, \quad T_j^{(3)}= T_j^{(3)}(\alpha, \lambda) \text{ real valued}.
\end{equation}
Moreover, since 
\begin{equation} \label{eq L+ Lambda2Q}
    -\frac{1}{2} L_+ (\Lambda \Lambda Q- 2\Lambda Q)= 2\phi_{Q \Lambda Q} \Lambda Q+ \phi_{(\Lambda Q)^2} Q+ 2\Lambda Q,
\end{equation} 
we have
\begin{equation}
    T_j^{(3)} = \frac{f^2}{2} (\Lambda \Lambda Q- 2\Lambda Q)- L_+^{-1} \sum_{k \neq j} \psi_{Q^2,k}^{(3)} Q. 
\end{equation}

For two functions $u$ and $v$, we write $u \equiv v$ if $u-v$ is an even function with exponential decay, and is $L^2$-orthogonal to all radial functions with exponential decay. By the proof of Proposition \ref{prop construction of approximate bubbles}, if we only want to determine the value of $b_j^{(n)}$ and $d_j^{(n)}$, then we can do calculations up to $\equiv$. 

By \cite[Lemma\;4.2]{GSW}, we have $\psi_{Q^2,k}^{(3)} Q \equiv 0$, and thus $L_+^{-1} \psi_{Q^2,k}^{(3)} Q \equiv 0$. This means
\begin{equation}
    T_j^{(3)} \equiv \frac{f^2}{2} (\Lambda \Lambda Q- 2\Lambda Q).
\end{equation}

Then we have 
\begin{equation}
    \re \ \hat{E}_j^{(3)}= -\lambda_j^4 \sum_{k \neq j} \Big( \psi_{Q^2,k}^{(2)} T_j^{(1)}+ 2\psi_{Q T_k^{(1)}}^{(2)} Q \Big)- \lambda_j b_j^{(2)} y_j T_j^{(1)}= 0
\end{equation}
and
\begin{equation} \begin{aligned}
    \imm \ \hat{E}_j^{(3)} &= \lambda_j m_j^{(2)} \Lambda T_j^{(1)}+ \lambda_j^2 \sum_{k=1}^m \bigg( \frac{\partial T_j^{(3)}}{\partial \alpha_k} \cdot 2\beta_k+ \frac{\partial T_j^{(1)}}{\partial \lambda_k} m_k^{(2)}\bigg) \\
    &= \lambda_j m_j^{(2)} \Lambda T_j^{(1)}+ \lambda_j^2 \frac{\partial T_j^{(1)}}{\partial \lambda_j} m_j^{(2)} + \lambda_j^2 \sum_{k=1}^m \frac{\partial T_j^{(3)}}{\partial \alpha_k} \cdot 2\beta_k \\
    &\equiv \lambda_j fh \Lambda \Lambda Q- 2\lambda_j fh \Lambda Q- \lambda_j fh (\Lambda \Lambda Q- \Lambda Q)= 0.
\end{aligned} \end{equation}
Since $\nabla \psi_{Q^2,k}^{(4)} \equiv 0$ in view of \cite[Lemma\;4.2]{GSW}, we take 
\begin{equation} \begin{gathered}
    b_j^{(4)}= d_j^{(4)}=0, \quad T_j^{(4)} \equiv - L_+^{-1}\sum_{k \neq j} \psi_{Q^2,k}^{(4)} Q.
\end{gathered} \end{equation}

Then we have
\begin{align}
    \imm \ \hat{E}_j^{(4)} &= \lambda_j^2 \sum_{k=1}^m \bigg( \frac{\partial \re \ T_j^{(4)}}{\partial \alpha_k} \cdot 2\beta_k+ \frac{\partial T_j^{(1)}}{\partial \alpha_k} \cdot 4\mu_k \alpha_k \bigg) \\
    &= \lambda_j^2 \sum_{k=1}^m \frac{\partial \re \ T_j^{(4)}}{\partial \alpha_k} \cdot 2\beta_k+ \lambda_j^2 \sum_{k=1}^m \Big( \frac{\partial f}{\partial \alpha_k} \cdot 4\mu_k \alpha_k \Big) \Lambda Q
\end{align}
and 
\begin{align}
    \re \ \hat{E}_j^{(4)} &=- \lambda_j^4 \Big( 2\phi_{QT_j^{(3)}} T_j^{(1)}+ 2\phi_{T_j^{(1)} T_j^{(3)}}Q+ 2\phi_{Q T_j^{(1)}} T_j^{(3)}+ \phi_{|T_j^{(1)}|^2} T_j^{(1)} \Big) \\
    &\quad \ -\lambda_j^4 \sum_{k \neq j} \Big( \psi_{Q^2,k}^{(3)} T_j^{(1)}+ 2\psi_{QT_k^{(1)}}^{(3)} Q+ \psi_{Q^2,k}^{(1)} T_j^{(3)} \\
    &\qquad \qquad \qquad+ 2\psi_{QT_k^{(3)}}^{(1)} Q+ \psi_{|T_k^{(1)}|^2}^{(1)} Q+ 2\psi_{QT_k^{(1)}}^{(1)} T_j^{(1)} \Big) \\
    &\equiv -\lambda_j^4 f^3 \Big( \phi_{Q( \Lambda \Lambda Q- 2\Lambda Q)} \Lambda Q+ \phi_{\Lambda Q( \Lambda \Lambda Q- 2\Lambda Q)} Q \\
    &\qquad \qquad \ \ + \phi_{Q \Lambda Q} (\Lambda \Lambda Q- 2\Lambda Q)+ \phi_{(\Lambda Q)^2} \Lambda Q+ (\Lambda \Lambda Q- 2\Lambda Q) \Big).
\end{align}
Using parity and $(\Lambda Q, Q)=0$, we have $(\re \ \hat{E}_j^{(4)}, \nabla Q)= (\imm \ \hat{E}_j^{(4)}, Q)=0$. Thus, we take
\begin{equation}
    b_j^{(5)}=d_j^{(5)}=0. 
\end{equation}
Moreover, since $\psi_{Q^2,k}^{(5)}Q \equiv0$ by \cite[Lemma\;4.2]{GSW} and
\begin{align}
    -L_+(\Lambda \Lambda \Lambda Q- 6\Lambda \Lambda Q+ 8\Lambda Q) &= \phi_{Q( \Lambda \Lambda Q- 2\Lambda Q)} \Lambda Q+ \phi_{\Lambda Q( \Lambda \Lambda Q- 2\Lambda Q)} Q \\
    &\quad + \phi_{Q \Lambda Q} (\Lambda \Lambda Q- 2\Lambda Q)+ \phi_{(\Lambda Q)^2} \Lambda Q+ (\Lambda \Lambda Q- 2\Lambda Q),
\end{align} 
we have
\begin{align}
    T_j^{(5)} \equiv &\ \frac{f^3}{6} (\Lambda \Lambda \Lambda Q- 6\Lambda \Lambda Q+ 8\Lambda Q)- i\lambda_j^{-2} \sum_{k=1}^m \Big( \frac{\partial f}{\partial \alpha_k} \cdot \mu_k \alpha_k \Big) |y_j|^2 Q \\
    &\ + i\lambda_j^{-2} \sum_{k=1}^m \frac{\partial L_-^{-1} \re \ T_j^{(4)}}{\partial \alpha_k} \cdot 2\beta_k.
\end{align}

Then we have
\begin{align}
    \imm \ \hat{E}_j^{(5)} &= -2\lambda_j^4 \phi_{QT_j^{(1)}} \imm \ T_j^{(4)}- \lambda_j^4 \sum_{k \neq j} \psi_{Q^2,k}^{(1)} \imm \ T_j^{(4)}+ \lambda_j m_j^{(2)} \Lambda T_j^{(3)} \\
    &\quad \ +\lambda_j^2 \sum_{k=1}^m \bigg( \frac{\partial \re \ T_j^{(5)}}{\partial \alpha_k} \cdot 2\beta_k- \frac{\partial T_j^{(1)}}{\partial \lambda_k} 4\mu_k \lambda_k+ \frac{\partial T_j^{(3)}}{\partial \lambda_k} m_k^{(2)} \bigg) \\
    &\equiv \lambda_j h \frac{f^2}{2} \Lambda( \Lambda \Lambda Q- 2\Lambda Q)- \lambda_j \frac{f^2}{2} h (\Lambda \Lambda \Lambda Q- 6\Lambda \Lambda Q+ 8\Lambda Q) \\
    &\quad \ + 8\lambda_j^2 \mu_j f \Lambda Q- 2\lambda_j f^2 h (\Lambda \Lambda Q- 2\Lambda Q) \\
    &= 8\lambda_j^2 \mu_j f \Lambda Q
\end{align}
and
\begin{align}
    \re \ \hat{E}_j^{(5)} &= -\lambda_j^4 \Big( 2\phi_{QT_j^{(1)}} \re \ T_j^{(4)}+ 2\phi_{Q \re T_j^{(4)}} T_j^{(1)}+ 2\phi_{T_j^{(1)} \re \ T_j^{(4)}} Q \Big) \\
    &\quad \ -\lambda_j^4 \sum_{k=1}^m \Big( \psi_{Q^2,k}^{(1)} \re \ T_j^{(4)}+ 2\psi_{Q \re \ T_k^{(4)}}^{(1)} Q+ \psi_{Q^2,k}^{(2)} T_j^{(3)}+ 2\psi_{QT_k^{(3)}}^{(2)} Q \\
    &\qquad \qquad \qquad+ \psi_{|T_k^{(1)}|^2}^{(2)} Q+ 2\psi_{Q T_k^{(1)}}^{(2)} T_j^{(1)}+ \psi_{Q^2,k}^{(4)} T_j^{(1)}+ \psi_{QT_k^{(1)}}^{(4)} Q \Big) \\
    &\quad  \ -\lambda_j b_j^{(2)} y_j T_j^{(3)}- \lambda_j^2 \sum_{k=1}^m \frac{\partial \imm \ T_j^{(5)}}{\partial \alpha_k} \cdot 2\beta_k \\
    &= - \lambda_j^2 \sum_{k=1}^m \frac{\partial \imm \ T_j^{(5)}}{\partial \alpha_k} \cdot 2\beta_k+ \text{odd function}.
\end{align}
Thus, we take $d_j^{(6)}=0$, $b_j^{(6)}$ properly and 
\begin{align}
    \re \ T_j^{(6)} &= \lambda_j^4 \sum_{l=1}^m \sum_{k=1}^m \frac{\partial}{\partial_{\alpha_l}} \Big( \frac{\partial f}{\partial \alpha_k} \cdot \mu_k \alpha_k \Big) \cdot 2\beta_l L_+^{-1} (|y_j|^2 Q)+ \text{odd function} \\
    &= 2\lambda_j^2 \| Q \|_{L^2}^2 \sum_{l=1}^m \frac{\partial}{\partial \alpha_l} \sum_{k \neq j} \frac{\alpha_{jk} \cdot \beta_{jk}}{|\alpha_{jk}|^4} \cdot \mu_l \alpha_l \rho+ \mathcal{A}+ \text{odd function},
\end{align}
where $\mathcal{A} \in S_8$. Then we have
\begin{align}
    \imm \ \hat{E}_j^{(6)} &= -2\lambda_j^4 \phi_{QT_j^{(1)}} \imm \ T_j^{(5)}- \lambda_j^4 \sum_{k \neq j} \Big( \psi_{Q^2,k}^{(1)} \imm \ T_j^{(5)}+ \psi_{Q^2,k}^{(2)} \imm \ T_j^{(4)} \Big) \\
    &\quad \ -\lambda_j b_j^{(2)} \cdot y_j \imm \ T_j^{(4)}+ \lambda_j m_j^{(2)} \Lambda \re \ T_j^{(4)} \\
    &\quad \ +\lambda_j^2 \sum_{k=1}^m \bigg( \frac{\partial \re \ T_j^{(6)}}{\partial \alpha_k} \cdot 2\beta_k+ \frac{\partial T_j^{(3)}}{\partial \alpha_k} \cdot 4\mu_k \alpha_k+ \frac{\partial \re \ T_j^{(4)}}{\partial \lambda_k} m_k^{(2)} \bigg) \\
    &\equiv -2\lambda_j^4 f (\phi_{Q\Lambda Q}+1) \imm \ T_j^{(5)}+ \lambda_j^2 \sum_{k=1}^m \frac{\partial T_j^{(3)}}{\partial \alpha_k} \cdot 4\mu_k \alpha_k \\
    &\quad \ + \lambda_j^2 \frac{\partial \re \ T_j^{(6)}}{\partial \alpha_k} \cdot 2\beta_k+ \text{odd function} \\
    &\equiv 2\lambda_j^2 f \sum_{k=1}^m \Big( \frac{\partial f}{\partial \alpha_k} \cdot \mu_k \alpha_k \Big) \big( \phi_{Q\Lambda Q} |y_j|^2 Q+ |y_j|^2 Q +2(\Lambda \Lambda Q- 2\Lambda Q) \big) \\
    &\quad \ + 4\lambda_j^4 \| Q \|_{L^2}^2 \sum_{s=1}^m \Big( \frac{\partial}{\partial \alpha_s} \sum_{l=1}^m \frac{\partial}{\partial \alpha_l} \sum_{k \neq j} \frac{\alpha_{jk} \cdot \beta_{jk}}{|\alpha_{jk}|^4} \cdot \mu_l \alpha_l \Big) \cdot \beta_s \rho+ \text{odd function}.
\end{align}
Since $L_-(\Lambda Q)= -2\phi_{Q \Lambda Q}Q- 2Q$ and $L_-(|y_j|^2Q)= -4\Lambda Q$, we have
\begin{align}
    \Big( \phi_{Q\Lambda Q} |y_j|^2 Q+ |y_j|^2 Q +2(\Lambda \Lambda Q- 2\Lambda Q), Q \Big)= \big( \Lambda Q, -\frac{1}{2} L_-(|y_j|^2 Q) \big)- 2(\Lambda Q, \Lambda Q)=0
\end{align}
by integration by parts. The odd term is also orthogonal to $Q$. Thus we take
\begin{equation}
    d_j^{(7)}= -4\lambda_j^2 \| Q \|_{L^2}^2 \sum_{s=1}^m \Big( \frac{\partial}{\partial \alpha_s} \sum_{l=1}^m \frac{\partial}{\partial \alpha_l} \sum_{k \neq j} \frac{\alpha_{jk} \cdot \beta_{jk}}{|\alpha_{jk}|^4} \cdot \mu_l \alpha_l \Big) \cdot \beta_s.
\end{equation}

\begin{Prop} \label{prop trajectory}
Let $(\alpha^\infty, \beta^\infty)$ be a hyperbolic-parabolic solution of \eqref{m-body}. Assume $\lambda^\infty \in \R_+^m$ and $\mu^\infty= \delta^\infty=0$. Then there exists a solution $P^{(N)}$ of \eqref{eq parameters} such that 
\begin{equation}
    \Big| P^{(N)}(t)- P^\infty(t) \Big| \to 0 \text{ as } t \to +\infty.
\end{equation}
For $(\alpha, \beta, \lambda)$ in $P^{(N)}$ we associate the cluster partition given by $(\alpha^\infty, \beta^\infty)$.
\end{Prop}

\begin{proof}
Take $\epsilon>0$ small. Define $Y= \Big\{ P \in C \big( [T_0,+\infty), \Omega \big) \ \big| \ {\| P-P^\infty \|}_1 \le 1 \Big\}$, where
\begin{equation}
    {\| P \|}_1:= \sum_{j=1}^m \sup_{t \ge T_0} \Big( t^{\f12-3\epsilon} |\alpha_j(t)|+ t^{\f32-3\epsilon} |\beta_j(t)|+ t^{\f32-3\epsilon} |\lambda_j(t)|+ t^{\f52-2\epsilon} |\mu_j(t)|+ t^{\f72-\epsilon} |\delta_j(t)| \Big).
\end{equation}

Note that if $\alpha_j^\infty$ and $\alpha_k^\infty$ are in the same cluster, then $\alpha_{jk}^\infty \sim t^{\f12}$ and $\beta_{jk}^\infty \sim t^{-\f12}$. If they are in different clusters, then $\alpha_{jk}^\infty \sim t$ and $\beta_{jk}^\infty \sim 1$. We have respectively
\begin{equation}
    \frac{\alpha_{jk}^\infty \cdot \beta_{jk}^\infty}{|\alpha_{jk}^\infty|^4} \sim t^{-2}\;\;\text{(same cluster)},\;\;\; \frac{\alpha_{jk}^\infty \cdot \beta_{jk}^\infty}{|\alpha_{jk}^\infty|^4} \sim t^{-3}\;\;\text{(different clusters)}.
\end{equation}
Thus $|d_j^{(7)}(P^\infty)| \lesssim t^{-\f92}$ and particularly for $P \in Y$ we deduce $|d_j^{(7)}(P)| \le t^{-\f92}$. By the Taylor formula, we have
\begin{equation}
    b_j^{(2)}(P)- b_j^{(2)}(P^\infty)= -\| Q \|_{L^2}^2 \sum_{k \neq j} \left( \frac{\alpha_{jk}- \alpha_{jk}^\infty}{|\alpha_{jk}^\infty|^4}- \frac{4\alpha_{jk}^\infty \cdot (\alpha_{jk}- \alpha_{jk}^\infty)}{|\alpha_{jk}^\infty|^6} \alpha_{jk}^\infty \right)+ O(t^{-\f72+ 6\epsilon}).
\end{equation}
Hence there exists $A_j \in \R^{4 \times 4m}$ such that
\begin{equation}
    b_j^{(2)}(P)- b_j^{(2)}(P^\infty)= \frac{A_j (\alpha- \alpha^\infty)}{t^2}+ O(t^{-3+ 3\epsilon}),
\end{equation}
where $\alpha$ is understood as a column vector. Now let us set $A=(A_1^T, \cdots, A_m^T)^T \in \R^{4m \times 4m}$, then \eqref{eq parameters} can be rewritten into
\begin{equation} \label{this-eq-traje}
    \left\{ \begin{aligned}
        &\dot{\alpha}- \dot{\alpha}^\infty= 2(\beta- \beta^\infty)+ O(t^{-\f32+ 2\epsilon}), \\
        &\dot{\beta}- \dot{\beta}^\infty= \frac{A (\alpha- \alpha^\infty)}{t^2} + O(t^{-\f52+ 2\epsilon}), \\
        &\dot{\lambda}= O(t^{-\f32+ 2\epsilon}), \ \dot{\mu}= O(t^{-\f72+ \epsilon}), \ \dot{\delta}= O(t^{-\frac{9}{2}}). 
    \end{aligned} \right.
\end{equation}
We then use the same argument as in \cite{GSW} for \eqref{this-eq-traje} to conclude the claim.
\end{proof}

We have the following more precise estimate: 
\begin{equation}
    |\alpha^{(N)}(t)- \alpha^\infty(t)| =o(t^{-\frac{1}{2}+}),\; |\beta^{(N)}(t)- \beta^\infty(t)|= o(t^{-\frac{3}{2}+}), \; |\lambda^{(N)}(t)- \lambda^\infty|=o(t^{-\frac{1}{2}+}),
\end{equation}
and
\begin{equation} \label{mixed N est}
    |\alpha^{(N)}(t) | \lesssim t,\; |\beta^{(N)}(t)| \lesssim 1, \; \lambda^{(N)}(t) \sim 1,\; |\mu^{(N)}(t)| \lesssim t^{-\frac{5}{2}},\; |\delta^{(N)}(t)| \lesssim t^{-\frac{7}{2}}.
\end{equation}
Further, taking differences in \eqref{this-eq-traje}, we have for any cluster $K$  and all $j , k \in K$
\begin{equation} \label{mixed N est - same cluster}
    |\alpha_j^{(N)}(t) - \alpha_k^{(N)}(t)   | \lesssim t^{\f12},\; |\beta_j^{(N)}(t) - \beta_k^{(N)}(t)| \lesssim t^{- \f12},\;|\lambda_j^{(N)}(t) - \lambda_k^{(N)}(t)| \lesssim t^{- \f12},
\end{equation}
where for the latter we make the assumption $\lambda_j^{\infty} = \lambda_k^{\infty} $ for all $K$ and $j,k \in K$. We note the implicit constants only depend on $P^\infty$.

\section{Bootstrap reduction}\label{sec:bootstrap}

We follow the strategy in \cite[Section 5]{GSW} (see also \cite{Wu}, \cite{KRM}) and reduce the proof of Theorem \ref{main-thm}  to a uniform backwards estimate.

\begin{Prop}[Uniform bound] \label{prop uniform estimate}
Let $P^{(N)}$ be the solution of \eqref{eq parameters} in Proposition \ref{prop trajectory} and $\gamma_j^{(N)}(t)$ be such that
\begin{equation}
    \gamma_j^{(N)}(0)=0, \ \dot{\gamma}_j^{(N)}= (\lambda_j^{(N)})^2- |\beta_j^{(N)}|^2- \big( \dot{\beta}_j^{(N)}+ 4\mu_j^{(N)} \beta_j^{(N)} \big) \cdot \alpha_j^{(N)}- \big( \dot{\mu}_j^{(N)}+ 4(\mu_j^{(N)})^2 \big) |\alpha_j^{(N)}|^2.
\end{equation}
For large $N$  there exists $T_0=T_0(N)$ such that if   $(T_n)_n$ is a sequence in $[T_0, + \infty)$ with $T_n \to +\infty$, we have the following. Let  $u_n$ be the solution to 
\begin{equation} \label{eq equation of un}
    \left\{ \begin{aligned}
        &i\partial_t u_n+ \Delta u_n- \phi_{|u_n|^2} u_n= 0, \\
        &u_n(T_n,\cdot)= R_{g^{(N)}} ^{(N)}(T_n,\cdot),
    \end{aligned} \right.
\end{equation}
with backwards maximal interval $(t_n, T_n]$ and $T_* \in [T_0,T_n] \cap (t_n, T_n]$. Then if
\begin{equation} \label{eq uniform estimate, bootstrap}
    \left\| u_n(t)- R_{g^{(N)}}^{(N)}(t) \right\|_{H^1} \le 2t^{-\frac{N}{9}}, \qquad \forall n\ge 1,\ \forall t \in [T_*,T_n],
\end{equation}
also
\begin{equation}
    \left\| u_n(t)- R_{g^{(N)}}^{(N)}(t) \right\|_{H^1} \le t^{-\frac{N}{9}}, \qquad \forall n\ge 1,\ \forall t \in [T_*,T_n].
\end{equation}
\end{Prop}

This proposition readily implies Theorem \ref{main-thm} as seen in \cite{GSW} and similarly in \cite{Wu}. For the proof of Theorem \ref{main-thm} by Proposition \ref{prop uniform estimate} we refer to \cite[Section 5]{GSW}.
\;\\[6pt]
In order to prove Proposition \ref{prop uniform estimate}, we deal with the zero modes of the linearized operators. The following lemma establishes a suitable modulation path $P$ \emph{with orthogonality conditions using the bootstrap assumption \eqref{eq uniform estimate, bootstrap}}.

\begin{Lemma} \label{lem orthogonality}
Let $N,n \ge 1$. Then there exist $T_0=T_0(N)>0$ and a unique modulation path $P$ given by parameters $g \in C^1([T_0,+\infty), \Omega \times (\R/2\pi\Z)^m)$ such that: if 
\begin{equation}
    \varepsilon(t,x)= u_n(t,x)- R_g^{(N)}(t,x),
\end{equation}
then for $t \ge T_0$ and $1\le j \le m$, we have
\begin{equation} \label{eq orthogonality} \begin{aligned}
    &\ \re \Big( \varepsilon(t), g_j Q \Big)= \re \Big( \varepsilon(t), g_j(x Q) \Big)= \re \Big( \varepsilon(t), g_j \big( |x|^2 Q \big) \Big) \\
    &= \imm \Big( \varepsilon(t), g_j \big( \Lambda Q \big) \Big)= \imm \Big( \varepsilon(t), g_j \big( \nabla Q \big) \Big)= \imm \Big( \varepsilon(t), g_j \rho \Big)= 0.
\end{aligned} \end{equation} 
In particular, we have
\begin{equation} \label{eq value of epsilon at T_n}
    g(T_n)=g^{(N)}(T_n), \quad \varepsilon(T_n)=0.
\end{equation}
\end{Lemma} 
 For a proof let us refer to a similar argument in \cite[Lemma 3.1]{Wu}. With $P$  as in the previous lemma we now reduce Proposition \ref{prop uniform estimate} to the modulation estimate in the following bootstrap argument.

\begin{Prop}[Bootstrap argument] \label{prop bootstrap}
For $N$ and $T_0=T_0(N)$ large enough, $\forall n\ge 1,\ T_* \in [T_0,T_n] \cap (t_n, T_n]$, if
\begin{equation} \label{eq bootstrap assumption}
   \left\{ \begin{aligned}
       &\| \varepsilon(t) \|_{H^1} \le t^{-\frac{N}{4}}, \quad \| x \varepsilon(t) \|_{L^2} \le t^{-\frac{N}{4}+2}, \\
       &\sum_{j=1}^m \left| \lambda_j(t)- \lambda_j^{(N)}(t) \right|+ \left| \beta_j(t)- \beta_j^{(N)}(t) \right| \le t^{-1-\frac{N}{8}},\\
       &\sum_{j=1}^m \left| \gamma_j(t)- \gamma_j^{(N)}(t) \right|+ \left| \alpha_j(t)- \alpha_j^{(N)}(t) \right| \le t^{-\frac{N}{8}},\\
       &\sum_{j=1}^m \left| \mu_j(t)- \mu_j^{(N)}(t) \right| \le t^{-\frac{5}{2} - \frac{N}{8}},\;\;\;\sum_{j=1}^m \left| \delta_j(t)- \delta_j^{(N)}(t) \right| \le t^{-\f72 - \frac{N}{8}}, 
   \end{aligned} \right. 
\end{equation}
for any $t \in [T_*,T_n]$, then
\begin{equation} \label{eq bootstrap conclusion}
   \left\{ \begin{aligned}
       &\| \varepsilon(t) \|_{H^1} \le \frac{1}{2} t^{-\frac{N}{4}}, \quad \| x \varepsilon(t) \|_{L^2} \le \frac{1}{2} t^{-\frac{N}{4}+2}, \\
       &\sum_{j=1}^m \left| \lambda_j(t)- \lambda_j^{(N)}(t) \right|+ \left| \beta_j(t)- \beta_j^{(N)}(t) \right| \le \frac{1}{2} t^{-1-\frac{N}{8}}, \\
       &\sum_{j=1}^m \left| \gamma_j(t)- \gamma_j^{(N)}(t) \right|+ \left| \alpha_j(t)- \alpha_j^{(N)}(t) \right| \le \frac{1}{2} t^{-\frac{N}{8}},\\
       & \sum_{j=1}^m \left| \mu_j(t)- \mu_j^{(N)}(t) \right| \le \frac{1}{2} t^{-\frac{5}{2} - \frac{N}{8}},\;\;\; \sum_{j=1}^m \left| \delta_j(t)- \delta_j^{(N)}(t) \right| \le \frac{1}{2}t^{-\f72 - \frac{N}{8}}    
   \end{aligned} \right. 
\end{equation}
for any $t \in [T_*,T_n]$.
\end{Prop}
\;\\
The proof of Proposition \ref{prop uniform estimate} now follows directly from Proposition \ref{prop bootstrap}, where we use that  \eqref{eq bootstrap assumption} is implied by a bootstrap argument and 
\[
\| u_n(t)- R_{g^{(N)}}^{(N)}\|_{H^1} \leq  \|\varepsilon(t)\|_{H^1} +  \|R_g^{(N)}- R_{g^{(N)}}^{(N)}\|_{H^1}. 
\]
Thus \eqref{eq bootstrap conclusion} applies to estimate the right side with $t \geq T_* \geq T_0(N)$ large enough.

\section{Modulation and error estimate}\label{sec:mod}

It remains to provide a proof of Proposition \ref{prop bootstrap} which is the purpose of this section. 

\subsection{Modulation}
We consider the  modulation path $P$ in Lemma \ref{lem orthogonality} and Proposition  \ref{prop bootstrap}.\\[4pt]
Let us start by the observation that from  \eqref{mixed N est} and \eqref{mixed N est - same cluster}, the bootstrap assumption \eqref{eq bootstrap assumption} in Proposition \ref{prop bootstrap} implies  for $ T_0 \gg1 $ on $ [T_0, T_n]$
\begin{align}\label{param-est-mix}
&|\alpha(t) | \lesssim t,\; |\beta(t)| \lesssim 1,\; |\lambda(t)| \sim 1,\; |\mu(t)| \lesssim t^{-\f52},\;|\delta(t)| \lesssim t^{-\f72},\\
\label{param-est-mix - same cluster}
&|\alpha_j- \alpha_k| \lesssim t^{\f12}, \;\; |\beta_j- \beta_k| \lesssim t^{-\frac{1}{2}}, \;\; |\lambda_j- \lambda_k| \lesssim t^{-\frac{1}{2}},
\end{align}
where the latter line is true for any cluster $K$ and all $j,k \in K$.\\[4pt]
\emph{\textbf{Step 1.}} \emph{Modulation estimate.}\;  We write $u = u_n$ and let the error for the modulation equations \eqref{eq parameters} (as in \cite[Section 6.1]{GSW}) be defined as
\begin{equation} \label{Mod-error}\begin{aligned} 
    Mod(t) := &\ \big| \dot{\alpha}_j- 2\beta_j- 4\mu_j \alpha_j \big|+ \big| \dot{\beta}_j+ 4\mu_j \beta_j+ (\dot{\mu}_j+ 4\mu_j^2) \alpha_j- B_j^{(N)} \big| \\
    &+ \big| \dot{\lambda}_j+ 4\lambda_j \mu_j- M_j^{(N)} \big| + \big| \dot{\mu}_j+ 4\mu_j^2- \lambda_j^4 \delta_j \big|+ \big| \dot{\delta}_j- D_j^{(N)} \big| \\
    &+ \big| \dot{\gamma}_j+ (\dot{\beta}_j+ 4\mu_j \beta_j) \cdot \alpha_j+ (\dot{\mu}_j+ 4\mu_j^2) |\alpha_j|^2+ |\beta_j|^2- \lambda_j^2 \big|,
\end{aligned} \end{equation}
where now $M_j^{(N)} = M_j^{(N)}(P),\;B_j^{(N)} = B_j^{(N)}(P),\; D_j^{(N)} = D_j^{(N)}(P)  $. Hence evaluating \eqref{H} in $u = \varepsilon + R$ with $R =R_g^{(N)}$ as in \eqref{eq approximate solution}, we obtain 
\begin{align} \label{eq equation of epsilon}
   & i\partial_t \varepsilon+ \Delta \varepsilon- \phi_{|R|^2} \varepsilon- 2\phi_{\re (\varepsilon \overline{R}) } R = \Psi+ \cN(\varepsilon)+ \sum_{j=1}^m S_j(y_j) e^{i(\gamma_j+ \beta_j \cdot x+ \mu_j |x|^2)},\\ \label{eq nonlin epsilon}
   & \cN(\varepsilon)= 2\phi_{\mathrm{Re} (\varepsilon \overline{R})} \varepsilon+ \phi_{|\varepsilon|^2}R+ \phi_{|\varepsilon|^2} \varepsilon,
\end{align}
where $\Psi= \Psi^{(N)}$ as defined in \eqref{eq definition of Psi} is the interaction error. 
\begin{Prop} \label{prop modulation estimate}
We have for all $t \in [T_*, T_n]$
\begin{equation} \label{eq modulation estimate}
Mod(t) \le \frac{C\| \varepsilon \|_{H^1}}{a^2}+ C_N\big(a^{- N -2} + \| \varepsilon \|_{H^1}^2\big).\;\;\;
\end{equation}
\end{Prop}
\begin{proof}
The proof is as in \cite[Section 6.1]{GSW}, hence let us give a sketch and spare details.\\[4pt]
We note for $ \cN(\varepsilon)$ in \eqref{eq equation of epsilon}  we have
$ |\cN(\varepsilon)| \lesssim \| \varepsilon \|_{H^1}^2 $ and in order to conclude \eqref{eq modulation estimate} we now rely on the orthogonality conditions in Lemma \ref{lem orthogonality}.\\[4pt]
First let $\theta_j= g_j \theta$ be the modulation of a smooth decaying function, more precisely we require 
\begin{align} \label{eq decay of theta}
  &\theta_j(y_j) = \lambda_j^2 \theta(\lambda_j(x - \alpha_j)) e^{i(\gamma_j + \beta_j \cdot x + \mu_j|x|^2)},\;\;|\nabla^k \theta(x)| \le C_k e^{-c_k|x|},\;\;x \in \mathbb{R}^4,\; k \geq 0. 
\end{align}
Considering \eqref{eq equation of epsilon} and taking $T_0(N) \gg 1 $ large enough we check that, after some careful calculation
\begin{equation} \begin{aligned}
    \frac{\d}{\dt} \bigg(\mathrm{Im} \int \varepsilon \overline{\theta_j}\; \d x \bigg)= &\ -\mathrm{Re} \int \varepsilon \lambda_j^2\overline{ (L_j \theta) e^{i\gamma_j+ i\beta_j \cdot x + i \mu_j |x|^2}}- 2\mathrm{Re} \int \varepsilon \phi_{\mathrm{Re} (\theta_j \overline{R}_j)} \sum_{k \neq j} \overline{R_k}\; \d x \\
    &+ \lambda_j^6 \mathrm{Re} \int S_j \overline{\theta}+ O_N \left( \frac{\| \varepsilon \|_{H^1}}{a^3}+ Mod \| \varepsilon \|_{H^1}+ \frac{1}{a^{N+2}}+ \| \varepsilon \|_{H^1}^2 \right)  
\end{aligned} \end{equation}
from Section \ref{sec:traje}, Proposition \ref{prop accuracy of approximate solution} and the definition of $Mod(t)$. Here we have set
\begin{equation}
    L_j \theta:= -\Delta \theta+ \theta+ 2\phi_{\mathrm{Re} (\theta \overline{V_j})} V_j+ \bigg( \phi_{|V_j|^2}+ \sum_{k \neq j} \psi_{|V_k|^2}^{(1)} \bigg) \theta. 
\end{equation}
Now let us choose $\theta$ to be either one of the functions  $ iQ,\; ix Q,\;i|x|^2 Q,\; \Lambda Q,\;\nabla Q, \;\rho $, then
\begin{align}
    &L_j \theta=f+ O\left( \frac{1}{a^2}+ \frac{C_N}{a^3} \right) \left( e^{-c |x-\alpha_j|}+ \frac{C_N}{a} e^{-c_N|x-\alpha_j|} \right),\\
    &\phi_{\mathrm{Re} (\theta_j \overline{R_j})}= O \left( \frac{1}{a^2}+ \frac{C_N}{a^3} \right) \big( 1+|x-\alpha_j| \big)^2,
\end{align}
where $f$ is such that $\re \ (\varepsilon, g_j f)=0$ using \eqref{eq orthogonality} and the latter expression  follows as in the proof of Proposition \ref{prop accuracy of approximate solution}. 
Further we check the lower bound 
\begin{equation} \label{eq estimate on theta: S_j}
    \sum_{\theta \in \Xi } \sum_{j=1}^m \left| \mathrm{Re} \int S_j(t,x) \overline{\theta}(t,x)\; \d x \right| \ge c\ Mod(t)- \frac{C_N}{a} Mod(t)
\end{equation}
from orthogonality. Combining these we get
\begin{equation}
    Mod(t) \le \frac{C\| \varepsilon \|_{H^1}}{a^2}+ \frac{C_N}{a^{N+2}}+ C_N \| \varepsilon \|_{H^1}^2,\;\;t \in [T_*, T_n],
\end{equation} 
 by taking $T_0(N) \gg 1$ again large enough to absorb all $O_N$ terms.
\end{proof}
\;\\
\emph{\textbf{Step 2.}} \emph{Integration of the modulation estimate.}\;\; We are now concerned with the second, third and fourth line of \eqref{eq bootstrap conclusion} assuming \eqref{eq bootstrap assumption} in Proposition \ref{prop bootstrap}.\\[4pt]
The proof here follows again in \cite[Section 6.1]{GSW}, however requires to check $M_j^{(N)}, B_j^{(N)}, D_j^{(N)}$ with the new admissibility condition in Section \ref{sec:approx} and the calculation in Section \ref{sec:traje}.\\[4pt]
First we note taking $ N \gg1 $ and $T_0 \gg1$ large we have from Proposition \ref{prop modulation estimate}
\[
Mod(t) \leq C t^{- \frac{N}{4}-1},\;\;\forall\; t \in [T_*, T_n].
\]
By Section \ref{sec:traje} each term in $ M_j^{(N)}, B_j^{(N)}$ is admissible of \emph{degree at least three} and we have 
\begin{align} 
&\quad \ \left | M_j^{(N)}(P) - M_j^{(N)}(P^{(N)})\right |  + \left | B_j^{(N)}(P) - B_j^{(N)}(P^{(N)})\right |\\
&\leq C\sum_{j=1}^m\Big(\frac{|\alpha_j - \alpha_j^{(N)}| }{a^4} + \frac{|\beta_j - \beta_j^{(N)}|+ |\lambda_j - \lambda_j^{(N)}|}{a^3} + a|\mu_j - \mu_j^{(N)} | + a^3|\delta_j - \delta_j^{(N)}| \Big),
\end{align}
where we note the observation that in Section \ref{sec:traje} the terms of degree three are explicit and independent of $\mu_j$ and $\delta_j$. By \eqref{eq bootstrap assumption} we then conclude
\begin{equation} \label{first bound mod int}
    \left | M_j^{(N)}(P) - M_j^{(N)}(P^{(N)})\right |  + \left | B_j^{(N)}(P) - B_j^{(N)}(P^{(N)})\right | \leq C t^{-2- \frac{N}{8}}.
\end{equation}
and further obtain from \eqref{eq parameters}, \eqref{first bound mod int} and \eqref{eq bootstrap assumption} as in \cite[Section 6.1]{GSW}
\begin{align} \label{this-bound-lambda}
&\quad \ |\dot{\lambda}_j - \dot{\lambda}^{(N)}_j | + |\dot{\beta}_j - \dot{\beta}^{(N)}_j | \\ \nonumber
&\leq\;  Mod(t) + \left | M_j^{(N)}(P) - M_j^{(N)}(P^{(N)})\right |  + \left | B_j^{(N)}(P) - B_j^{(N)}(P^{(N)})\right |\\ \nonumber
&\;\;\;\;+ | \lambda_j \mu_j  - \lambda_j^{(N)} \mu_j^{(N)}| +  | \mu_j \beta_j  - \mu_j^{(N)} \beta_j^{(N)}|+ |\lambda_j^4 \delta_j \alpha_j   - (\lambda_j^{(N)})^4 \delta_j^{(N)} \alpha_j^{(N)} | \leq C t^{- \frac{N}{8} -2}.
\end{align} 
Integrating and using \eqref{eq value of epsilon at T_n} we hence infer the claim in \eqref{eq bootstrap conclusion} after taking $N \gg1 $ large. 
Next, from Section \ref{sec:traje}, we know that the first non-vanishing term in $D_j^{(N)}$ is $d_j^{(7)}$ and $|d_j^{(7)}| \lesssim t^{-\f92}$ by the proof of Proposition \ref{prop trajectory}. From this and the admissibility condition we have
\begin{align}
&\quad \ \left | D_j^{(N)}(P) - D_j^{(N)}(P^{(N)}) \right | \\
&\leq C\sum_{j=1}^m\Big(\frac{|\alpha_j - \alpha_j^{(N)}| }{a^9} + \frac{|\beta_j - \beta_j^{(N)}|  + |\lambda_j - \lambda_j^{(N)} |}{a^8} + \frac{|\mu_j - \mu_j^{(N)} |}{a^4} + \frac{|\delta_j - \delta_j^{(N)} |}{a^2} \Big),
\end{align}
and hence we conclude
\begin{equation} \label{this-bound-delta}
    |\dot{\delta}_j - \dot{\delta}^{(N)}_j | \leq Mod(t)+ \left | D_j^{(N)}(P) - D_j^{(N)}(P^{(N)}) \right | \le C t^{-\frac{9}{2}- \frac{N}{8}}
\end{equation}
which we again integrate. Then, directly using  \eqref{eq parameters}, the bound for $Mod(t)$ and integrated versions of \eqref{this-bound-delta}, \eqref{this-bound-lambda}, we infer
\begin{align}
& | \dot{\alpha}_j - \dot{\alpha}_j^{(N)}| \leq C t^{-1- \frac{N}{8}},\;\; | \dot{\mu}_j - \dot{\mu}_j^{(N)}| \leq C t^{-\frac{7}{2}- \frac{N}{8}}.
\end{align} 
which likewise implies the claim in \eqref{eq bootstrap conclusion} for $\alpha_j, \mu_j$ taking $ N \gg1 $ large.
Finally, for the phase $\gamma_j$ we collect the above estimates and infer
\begin{align}
|\dot{\gamma}_j - \dot{\gamma}_j^{(N)}| &\leq C\Big( |\lambda_j - \lambda_j^{(N)} | + |\beta_j - \beta_j^{(N)} | +a |\dot{\beta}_j - \dot{\beta}_j^{(N)} | + \frac{|\alpha_j - \alpha_j^{(N)}|}{a^2} \\
&\qquad \;\; + a |\mu_j \beta_j- \mu_j^{(N)} \beta_j^{(N)}|+ a^2 |\dot{\mu}_j- \dot{\mu}_j^{(N)}| \Big) + Mod(t) \leq   C t^{-1- \frac{N}{8}}.
\end{align}

\subsection{Error estimate} Here we give a proof of the first line of the bootstrap estimate \eqref{eq bootstrap conclusion}.\\[5pt]
The argument follows \cite[Section 6.2]{GSW}, \cite{KRM}, \cite{Wu} and we first observe the by  expanding the energy via $ u = R + \varepsilon$ we have the conservation law\\
\begin{align}\label{energy-cons}
    2\mathcal{E}(u_0) =& \;\;2 \mathcal{E}(R) -2 \mathrm{Re} (\varepsilon, \overline{\Delta R - \phi_{|R|^2} R}) + \G_1(\varepsilon)
\end{align}
where the nonlinear part reads
\begin{align}
    \G_1(\varepsilon) =&\; \int |\nabla \varepsilon|^2+ \int \phi_{|R|^2} |\varepsilon|^2- 2\kappa \int |\nabla \phi_{\mathrm{Re} (\varepsilon \overline{R}) }|^2+ 2\int \phi_{\mathrm{Re} (\varepsilon \overline{R}) } |\varepsilon|^2- \frac{\kappa}{2} \int |\nabla \phi_{|\varepsilon|^2} |^2.
\end{align}
For deriving \eqref{eq bootstrap conclusion} it is not enough to estimate the interactions  $\G_1(\varepsilon)$ in \eqref{energy-cons} and thus we add localizations of mass, momentum, their center and the variance along $ \alpha(t)$ to $\G_1(\varepsilon)$. First let us note the following and refer to a similar proof in \cite[Section 4.2]{Wu}.
\begin{Lemma}  \label{lem cutoff}
There exist $c,C>0$ and $\varphi_j \in C^{1,\infty} (\mathbb{R}_+ \times \mathbb{R}^4)$ for $1 \le j \le m$ such that
\begin{equation}\label{eq cutoff} \begin{gathered}
    0 \le \varphi_j(t,x) \le 1, \quad \sum_{j=1}^m \varphi_j(t,x) \equiv 1,\;\;
    |\partial_t \varphi_j|+ |\nabla \varphi_j| \le \frac{C}{a}, \quad |\partial_t \sqrt{\varphi_j}|+ |\nabla \sqrt{\varphi_j}| \le \frac{C}{a},\\
    \varphi_j(t,x)= \left\{ \begin{aligned}
        &1, \quad |x-\alpha_j(t)| \le ca(t), \\
        &0, \quad |x-\alpha_k(t)| \le ca(t),\ k \neq j.
    \end{aligned} \right.
\end{gathered} \end{equation}
 Moreover, for any cluster $K$ we have
\begin{equation} \label{eq cutoff, mixed case} \begin{gathered}
    |\partial_t \varphi_K|+ |\nabla \varphi_K| \le Ct^{-1}, \quad \text{where } \varphi_K= \sum_{j \in K} \varphi_j.
\end{gathered} \end{equation}
\end{Lemma}
The cut-off functions $\varphi_j$ localize the approximate solution $R$ in \eqref{eq approximate solution} of Section \ref{sec:approx} to $R_j$. To be precise, we have, using Lemma \ref{lem properties of admissible functions-2} for some $c> 0$,
\begin{align}\label{est local varphi}
    \sup_{x \in \R^4}| \varphi_j(t)R(t) - R_j(t) | \leq C_N e^{- c a(t)},\;\; t \geq T_0.
\end{align}
We now define (see \cite[Section 6.2]{GSW}) the functional 
$
\displaystyle \G(\varepsilon)= \sum_{k=1}^5\G_k(\varepsilon),
$
where
\begin{align*}
    \G_2(\varepsilon)=&\; \sum_{j=1}^m \Big( \lambda_j^2+ |\beta_j|^2 \Big) \int \varphi_j |\varepsilon|^2,\;\;\; \G_3(\varepsilon)=\; -2\sum_{j=1}^m \beta_j \int \varphi_j \mathrm{Im} (\nabla \varepsilon \overline{\varepsilon}),\\[3pt]
    \G_4(\varepsilon) =&\; 4 \sum_{j=1}^m \mu_j^2 \int \varphi_j  |x|^2|\varepsilon|^2, \;\;\; \G_5(\varepsilon) =\; 4 \sum_{j=1}^m \mu_j \beta_j \int \varphi_j x |\varepsilon|^2 - 4 \sum_{j=1}^m \mu_j \int \varphi_j \mathrm{Im} (x \nabla \varepsilon \overline{\varepsilon}).
\end{align*}
and note the latter $\G_4, \G_5$ are specific to the $L^2$-critical case, i.e. to $d=4$ dimensions (c.f. \cite{Wu}).\\[5pt]
\emph{Remaining steps}. In order to conclude the first line of \eqref{eq bootstrap conclusion}, we need coercivity for $\G$ in Proposition \ref{prop coercivity} which  will be combined  with an upper bound in Proposition \ref{prop estimate on G(epsilon)} assuming the bootstrap condition \eqref{eq bootstrap assumption}.
\begin{Prop} \label{prop coercivity}
Let $N \ge 2$. For $T_0=T_0(N) \gg 1$ large enough, there exists $c_0>0$ such that 
\begin{equation}
    \G(\varepsilon(t)) \ge c_0 \| \varepsilon(t) \|_{H^1}^2,\;\; t \in [T_*,T_n].
\end{equation}

\end{Prop}
\begin{proof}[Sketch of proof]
The proof follows as provided below \cite[Proposition 6.3]{GSW}  and relies on Lemma \ref{non-deg-coer-Inv} part (2) in Section \ref{sec:gs-properties} combined with the orthogonality in \eqref{eq orthogonality} of Lemma \ref{lem orthogonality}.\\[4pt]
Let $\varepsilon_j= \varepsilon \sqrt{\varphi_j}$ and $\tilde{\varepsilon}_j= g_j^{-1} \varepsilon_j$, that is we define
\begin{equation}
    \tilde{\varepsilon}_j(t,y_j)= \frac{1}{\lambda_j^2(t)} \varepsilon_j(t, \lambda_j^{-1}(t)y_j+ \alpha_j(t)) e^{- i(\gamma_j(t) + \beta_j(t) \cdot (\lambda_j^{-1}y_j+ \alpha_j(t)) + \mu_j |\lambda_j^{-1}y_j+ \alpha_j(t)|^2)}.
\end{equation}
For $T_0(N) \gg1$ large enough we then have
\begin{equation}\label{this-coer-bound}
    \Big( L_+ \mathrm{Re}(\tilde{\varepsilon}_j), \mathrm{Re}(\tilde{\varepsilon}_j) \Big)+ \Big( L_- \mathrm{Im}(\tilde{\varepsilon}_j), \mathrm{Im}(\tilde{\varepsilon}_j) \Big) \ge c \| \tilde{\varepsilon}_j \|_{H^1}^2, \quad \forall t\ge T_0
\end{equation}
and define the truncated functionals
\begin{equation} \begin{aligned}
    \H_{j,\varphi}(\varepsilon) = &\ \int \varphi_j|\nabla \varepsilon|^2+ \int \phi_{|R_j|^2} |\varepsilon|^2- 2 \kappa \int \big| \nabla \phi_{\mathrm{Re}(\varepsilon \overline{{R_j}})} \big|^2 \\
    &\ + \Big( \lambda_j^2+ |\beta_j|^2 \Big) \int \varphi_j |\varepsilon|^2 - 4 \mu_j \int \varphi_j \mathrm{Im}  (x \nabla \varepsilon \overline{\varepsilon}) - 2\beta_j \int \varphi_j  \mathrm{Im} (\nabla \varepsilon \overline{\varepsilon})\\
    &\ + 4 \mu_j \beta_j \int \varphi_j x |\varepsilon|^2 + 4 \mu_j^2 \int \varphi_j |x|^2 |\varepsilon|^2.
\end{aligned} \end{equation}
Therefore we obtain, after some calculations (for details see the proof of \cite[Proposition 6.3]{GSW}), using \eqref{this-coer-bound},  Lemma \ref{lem properties of admissible functions-2} and  \eqref{eq cutoff},\eqref{est local varphi}
\begin{equation} \label{eq coercivity of H_j,phi} \begin{aligned}
    \H_{j,\varphi}(\varepsilon) \geq   c\int \varphi_j ( |\nabla \varepsilon|^2+ |\varepsilon|^2)- \frac{C_N}{a} \| \varepsilon \|_{H^1}^2.
\end{aligned} \end{equation}
Further by Lemma \ref{lem cutoff} we can rewrite 
\begin{equation} \begin{aligned}
    \G(\varepsilon)= &\ \sum_{j=1}^m \H_{j,\varphi}(\varepsilon)+ 2\int \phi_{\mathrm{Re} (\varepsilon \overline{R})} |\varepsilon|^2- \frac{\kappa}{2} \int |\nabla \phi_{|\varepsilon|^2}|^2 \\
    &\ + \sum_{j \neq k} \int \phi_{\mathrm{Re} (R_k \overline{R_j})} |\varepsilon|^2- 2 \kappa \sum_{j \neq k} \int \nabla \phi_{\mathrm{Re} (\varepsilon \overline{R_j})} \cdot \nabla \phi_{\mathrm{Re} (\varepsilon \overline{R_k})}. 
\end{aligned} \end{equation}
The first term is thus estimated by \eqref{eq coercivity of H_j,phi}, whereas the other terms in the first line are $O \big( t^{-N/4} \| \varepsilon \|_{H^1}^2 \big)$ by Proposition \ref{prop accuracy of approximate solution} and the bootstrap assumption \eqref{eq bootstrap assumption}. The two terms in the second line are $O_N \big( e^{-ca}\| \varepsilon \|_{H^1}^2 \big)$ and $O_N \Big( \frac{\| \varepsilon \|_{H^1}^2}{a} \Big)$, respectively.  This shows the claim after taking  $T_0(N) \gg1 $ large enough.
\end{proof}

\begin{Prop} \label{prop estimate on G(epsilon)}
Let us assume \eqref{eq bootstrap assumption} in Proposition \ref{prop bootstrap}. For $ N \geq 2$ and $T_0=T_0(N) \gg1$ large enough, there exists $C, C_N >0$ such that
 \begin{align}\label{G(epsilon) estimate}
     & |\G(\varepsilon(t))|   \leq \frac{C}{N} t^{- \frac{N}{2}} + C_N t^{- \frac{3N}{4}}.
 \end{align}
for all $t \in [T_*,T_n]$.
\end{Prop}

\begin{proof} We omit some details for the following  calculation found in the proof of \cite[ Proposition 6.4]{GSW}.  First using integration by parts we have for $T_0(N) \gg1 $ large  by \eqref{eq equation of epsilon} and Proposition \ref{prop accuracy of approximate solution}
\begin{equation} \begin{aligned}
    \frac{\d \G_1}{\dt}(\varepsilon) &=  O \big( Mod(t) \| \varepsilon \|_{H^1} \big)+ O_N \left( \frac{1}{a^{N+2}} \| \varepsilon \|_{H^1} \right)\\
    &\;\; +4\mathrm{Re} \int \phi_{\mathrm{Re} (\varepsilon \overline{R})} \varepsilon \partial_t \overline{R}+ 2\int \phi_{\mathrm{Re} (\partial_t R \overline{R})} |\varepsilon|^2+ 2\mathrm{Re} \int \phi_{|\varepsilon|^2} \varepsilon \partial_t \overline{R}, 
\end{aligned} \end{equation}
where we note $ |S_j(t,x)| \leq C Mod(t) e^{- c_N |x - \alpha_j|}$. Further since (for $T_0(N) \gg1 $ large)
\begin{align}
    & |V_j^{(N)}(y_j)| \leq C e^{- c|y_j|} + C_N a^{-1}e^{- c_N|y_j|},\\
    &|R_{g,j}^{(N)}(x)| \leq C e^{- c|x - \alpha_j|} + C_N a^{-1} e^{- c|x - \alpha_j|} ,\\[3pt]
    & |M_j^{(N)}| + |B_j^{(N)}| + |D_j^{(N)}| \leq  C a^{-2} + C_N a^{-3},
\end{align}
the estimate \eqref{param-est-mix} implies the following for $R_j = g_j V_j$
\begin{equation} \begin{aligned}
    \partial_t R_j =  -2\beta_j \cdot \nabla R_j+ i\big( \lambda_j^2 + |\beta_j|^2 \big) R_j + \left( O_N \Big(\frac{1}{a^3}\Big) +  O\Big(Mod + \frac{C}{t}\Big) \right) e^{-c_N|x-\alpha_j|}.
\end{aligned} \end{equation}
Hence applying  \eqref{eq modulation estimate} in Proposition \ref{prop modulation estimate}, we infer
\begin{align} \label{eq estimate of G_1} 
    \frac{\d \G_1}{\dt}(\varepsilon) = \sum_{j=1}^m \Bigg( 4\Big( & \lambda_j^2 + |\beta_j|^2 \Big) \int \phi_{\mathrm{Re}( \varepsilon \overline{R})} \mathrm{Im} (\varepsilon \overline{R_j}) - 8\int \phi_{\mathrm{Re} (\varepsilon \overline{R})} \mathrm{Re} (\varepsilon \beta_j \cdot \nabla \overline{R_j})\\ \nonumber
    &- 4\int \phi_{\mathrm{Re} (\beta_j \cdot \nabla R_j \overline{R_j})} |\varepsilon|^2 \Bigg)+ O \left( \frac{\| \varepsilon \|_{H^1}^2} {t} \right)+ O_N \left( \frac{\| \varepsilon \|_{H^1}}{a^{N+2}}+ \| \varepsilon \|_{H^1}^3 \right),
\end{align} 
where we used Hardy-Littlewood-Sobolev's estimate and a straight forward adaption of the distance estimates in \cite[Lemma 2.3]{Wu} for $d =4$ dimensions.\\[4pt]
Moreover by integration by parts, \eqref{param-est-mix}, \eqref{eq equation of epsilon} and  the bound for  $S_j$ as above, we likewise have
\begin{equation} \label{eq intermediate G2}\begin{aligned}
    \frac{\d \G_2}{\dt}(\varepsilon) &= \sum_{j=1}^m \Big( \lambda_j^2+ |\beta_j|^2 \Big) \big[ 4\int \varphi_j \mathrm{Im} (\phi_{\mathrm{Re}(\varepsilon \overline{R})} R \overline{\varepsilon}) + \int \Big( \partial_t \varphi_j |\varepsilon|^2+ 2\nabla \varphi_j \mathrm{Im}(\nabla \varepsilon \overline{\varepsilon}) \Big) \big]\\
    &\;  +  O_N \left( \frac{\| \varepsilon \|_{H^1}}{a^{N+1}}+ \| \varepsilon \|_{H^1}^3 \right) + O( Mod(t) \| \varepsilon \|_{H^1}),
\end{aligned} \end{equation}
where we recall $|\Psi|\lesssim a^{-N-2}, \mathcal{N}(\varepsilon) \lesssim \|\varepsilon\|_{H^1}^2$ implying the first term in the second line. Let us note by the modulation estimate \eqref{eq modulation estimate} and the localization property of $\varphi_j$ in \eqref{est local varphi} we obtain 
\begin{equation} \label{eq G_3 cutoff error} \begin{aligned}
     &\sum_{j=1}^m \Big( \lambda_j^2+ |\beta_j|^2 \Big) \big[ \int \varphi_j \mathrm{Im} (\phi_{\mathrm{Re}(\varepsilon \overline{R})} R \overline{\varepsilon})  + O( Mod(t) \| \varepsilon \|_{H^1})\\
     &=\; - \sum_{j=1}^m \Big( \lambda_j^2+ |\beta_j|^2 \Big)  \int \phi_{\mathrm{Re}(\varepsilon \overline{R})}\mathrm{Im} ( \overline{R}_j \varepsilon)  + O\big(\frac{\| \varepsilon \|_{H^1}}{a^{N+2}} + \| \varepsilon \|_{H^1}^3\big) + O_N\big(\frac{\| \varepsilon \|_{H^1}^2}{t}\big).
\end{aligned} \end{equation}
For the remaining terms in \eqref{eq intermediate G2} we consider the clusters $K$  and select (arbitrary) $\lambda_{K},  \beta_K $ such that $\lambda_{j}^2 = \lambda_K^2 + O(t^{- \f12})$ and  $|\beta_j|^2 = |\beta_K|^2 + O(t^{- \f12})$ for all $j \in K$ by \eqref{param-est-mix - same cluster}. Then
\begin{align}
&\sum_{j=1}^m \Big( \lambda_j^2+ |\beta_j|^2 \Big) \int \Big( \partial_t \varphi_j |\varepsilon|^2+ 2\nabla \varphi_j \mathrm{Im}(\nabla \varepsilon \overline{\varepsilon}) \Big)\\ \nonumber
&= \sum_{i =1}^l \Big( \lambda_{K_i}^2+ |\beta_{K_i}|^2 \Big) \int \Big( \partial_t \varphi_{K_i} |\varepsilon|^2+ 2\nabla \varphi_{K_i} \mathrm{Im}(\nabla \varepsilon \overline{\varepsilon}) \Big) + O_N\big(\frac{\| \varepsilon \|_{H^1}^2}{t}\big)\\
&= O_N\big(\frac{\| \varepsilon \|_{H^1}^2}{t}\big),
\end{align}
where we used $|\partial_t \varphi_j|+ |\nabla \varphi_j|= O(t^{-1/2})$ in the second line and  \eqref{eq cutoff, mixed case} to conclude the third line. 
Similarly we compute (for more details see \cite[Section 6.2]{GSW})
\begin{equation} \label{eq estimate of G_3} \begin{aligned}
    \frac{\d \G_3}{\dt} &=  \sum_{j=1}^m \left( 8\int \phi_{\mathrm{Re} (\varepsilon \overline{R})} \mathrm{Re}(\varepsilon \beta_j \cdot \nabla \overline{R_j})+ 4\int \phi_{\mathrm{Re} (\beta_j \cdot \nabla R_j \overline{R_j})} |\varepsilon|^2 \right) \\
    &\;+ \sum_{j=1}^m \beta_j \int \bigg( \nabla \varphi_j \Big( 2|\nabla \varepsilon|^2+ 2\phi_{\mathrm{Re} (\varepsilon \overline{R})} \mathrm{Re} ( \varepsilon \overline{R}) + \phi_{|R|^2} |\varepsilon|^2 \Big) + \partial_t \varphi_j \mathrm{Im} (\nabla \varepsilon \overline{\varepsilon}) \bigg)\\  
    &\; +O\left( \frac{\| \varepsilon \|_{H^1}^2}{t} \right)+ O_N \left( \frac{\| \varepsilon \|_{H^1}}{a^{N+2}}+ \frac{\| \varepsilon \|_{H^1}^2} {a^2}+ \| \varepsilon \|_{H^1}^3 \right).
\end{aligned} \end{equation}
Then the second line is of order $O\left( \frac{\| \varepsilon \|_{H^1}^2}{t} \right)$ by the cluster decomposition for $\beta_j$ as in $\mathcal{G}_2$ above, i.e. we note by \eqref{param-est-mix - same cluster}, \eqref{eq cutoff} and \eqref{eq cutoff, mixed case} we have 
\begin{equation}
 |\beta_K|\Big( |\partial_t \varphi_K|+ |\nabla \varphi_K| \Big) +  \max_{j \in K}\Big[ |\beta_j- \beta_K| \cdot \Big( |\partial_t \varphi_j|+ |\nabla \varphi_j| \Big)\Big] \le \frac{C}{t}
\end{equation}
for any fixed representative $\beta_K$ of the $\beta_j$'s of a cluster $ K$.
Combining the previous calculations, we hence deduce using the  bootstrap assumption \eqref{eq bootstrap assumption} and integration
\begin{align} \label{eq estimate of G_123}
    |\G_1(\varepsilon)+ \G_2(\varepsilon)+ \G_3(\varepsilon)| \le& \int_t^{T_n} C \tau^{- \frac{N}{2} -1} + C_N \tau^{-\frac{3N}{4} - 1}\; \d \tau\\ \nonumber
    \le& \frac{C}{N} t^{- \frac{N}{2}} + C_N t^{-\frac{3N}{4}}.
\end{align}
For the estimates of  $\mathcal{G}_4(\varepsilon), \mathcal{G}_5(\varepsilon)$, we refer to  (4) and (5) of the proof of \cite[Proposition 6.4]{GSW}  and the fact $|\mu_j(t)| \lesssim t^{- \f52}, |\beta_j(t)|\lesssim 1$. Precisely, via applying  again \eqref{eq definition of S_j^N}, \eqref{eq modulation estimate}, \eqref{eq equation of epsilon} and Proposition \ref{prop accuracy of approximate solution} to $ \frac{d}{dt} \int |x \varepsilon |^2$ for estimating $|\mathcal{G}_4(\varepsilon)|$ and using Cauchy-Schwarz for $\mathcal{G}_5(\varepsilon)$, we obtain
\begin{align} \label{This bound}
  |\mathcal{G}_4(\varepsilon)| \leq \frac{C}{N} t^{- \frac{N}{2}} + C_N t^{- \frac{3N}{4}},\;\;  |\mathcal{G}_5(\varepsilon)| \leq \frac{C}{N} t^{- \frac{N}{2}}.
\end{align}
This completes the proof.
\end{proof}

The final proof of the bootstrap Proposition \ref{prop bootstrap} is now a combination of the above Propositions.
\begin{proof}[Proof of \eqref{eq bootstrap conclusion}] It remains to check the first line. Note that by (6.18) in \cite[Section 6.2]{GSW}, which is the same upper bound of $|\mathcal{G}_4(\varepsilon)|$ in \eqref{This bound} but for $\| x \varepsilon \|_{L^2}^2$, we get the bound for  $\| x \varepsilon \|_{L^2}$ in \eqref{eq bootstrap conclusion}.  Then combining \eqref{G(epsilon) estimate} in Proposition \ref{prop estimate on G(epsilon)} and the coercivity of $\G(\varepsilon)$ in Proposition \ref{prop coercivity}, we take $ N $ and $T_0 = T_0(N)$ subsequently large to conclude the claim.
\end{proof}

\bibliographystyle{alpha}

\begin{thebibliography}{99}
		\small
		\bibitem{Caz} Cazenave, T. \emph{Semilinear Schr\"odinger equations}, Courant Lecture Notes in Math., Vol. 10,  NYU Courant Instit. of Math.Sci.,
		New York, 2003.
        \bibitem{CMM} Côte, R., Martel, Y., and Merle, F. \emph{ Construction of multi-soliton solutions for the $L^2$-supercritical gKdV and NLS equations}, Revista Matematica Iberoamericana, 27(1), 2011, pp 273-302.
		\bibitem{El-Schlein} Elgart, A. and Schlein, B. \emph{Mean field dynamics of boson stars},  Commun. Pure Appl. Math., Vol. 60, No. 4, 2007, p 500–545 
        \bibitem{Fan} Fan, C. \emph{log–log blow up solutions blow up at exactly m points}, Annales de l'Institut Henri Poincaré C, Analyse non linéaire. Vol. 34. No. 6. 2017, pp. 1429-1482
	    \bibitem{Fr-Lenz} Fr\"ohlich, J. and Lenzmann, E. \emph{Mean-field limit of quantum Bose gases and nonlinear
		Hartree equation}, S\'eminaire: \'Equations aux D\'eriv\'ees Partielles, p 2003–2004, S\'emin.
		\'Equ. D\'eriv. Partielles, \'Ecole Polytech., Palaiseau, 2004, pp. Exp. No. XIX, 26.
		\bibitem{Ginibre-Velo} Ginibre, J. and Velo, G. \emph{On a class of nonlinear Schr\"odinger equations with nonlocal
		interaction},  Math. Z., Vol. 170, No. 2, 1980, p 109–136.
        \bibitem{GSW} Gómez, J., Schmid T., and Wu Y. \emph{Multisoliton solutions and blow up for the $L^2$-critical Hartree equation}, to appear in Arch. Ration. Mech. Anal., arXiv:2501.18398 (2025). 
        \bibitem{J2} Jendrej, J. and Lawrie, A. \emph{Classification of kink clusters for scalar fields in dimension $1+1$}, arXiv preprint, arXiv:2412.16274 (2024).
        \bibitem{J1} Jendrej, J., Kowalczyk, M. and Lawrie, A. \emph{Dynamics of strongly interacting kink-antikink pairs for scalar fields on a line}. Duke Math. Journal,  171.18, 2022, pp 3643-3705.
		\bibitem{Koch-T} Koch, H. and Tataru, D. \emph{Multisolitons for the cubic NLS in 1-d and their stability}, Publications mathématiques de l'IHÉS, 2024, p 1-116.
		\bibitem{KLR} Krieger, J., Lenzmann, E. and Raphaël, P. \emph{On Stability of Pseudo-Conformal Blowup for $L^2$-critical Hartree NLS}, Annales Henri Poincaré, Vol. 6, No. 10, 2009, p 1159-1205.
		\bibitem{KRM} Krieger, J., Raphaël, P. and Martel, Y.  \emph{Two‐soliton solutions to the three‐dimensional gravitational Hartree equation},  Commun. Pure Appl. Math., Vol.  62, No. 11, 2009, p 1501-1550.
		\bibitem{Lieb} Lieb, E. H. \emph{Existence and uniqueness of the minimizing solution of Choquard’s non-
		linear equation},  Stud. Appl. Math., Vol. 57, No. 2, 1976/77, p 93–105.
        \bibitem{LenzmannNondegen} Lenzmann, E. \emph{Uniqueness of ground states for pseudorelativistic Hartree equations}, Analysis \& PDE, Vol. 2, No. 1, 2009, doi.org/10.2140/apde.2009.2.1
        \bibitem{Marchal-Saari} Marchal, C. and  Saari, D.G. \emph{On the final evolution of the n-body problem},  Journal of differential equations, Vol. 20, No.1,  1976, p 150-186.
		\bibitem{MM} Martel, Y. and Merle, F. \emph{ Multi solitary waves for nonlinear Schrödinger equations},  Ann. Inst. H.
		Poincaré Anal. Non Linéaire, Vol. 23 , No. 6, 2006, p 849–864.
		\bibitem{MR} Martel, Y. and  Raphaël, P. \emph{Strongly interacting blow up bubbles for the mass critical NLS}, Ann. Sci. Éc. Norm. Supér., Vol. 51, No. 4, 2018, p 701–737.
		\bibitem{Merle-k-bl} Merle, F.  \emph{Construction of solutions with exactly k blow-up points for the Schr\"odinger equation with critical nonlinearity}, Comm. Math. Phys., Vol. 129, No. 2, 1990, p 223–240.
        \bibitem{Mer-Raph} Merle, F. and Rapha\"el, P. \emph{ Profiles and quantization of the blow up mass for critical nonlinear Schr\"odinger equation}.  Comm. Math. Phys., 253(3), 2005, pp 675-704.
        \bibitem{N} Nguyen, T.V. \emph{Existence of multi-solitary waves with logarithmic relative distances for the NLS equation}. Comptes Rendus Mathematique 357.1, 2019, pp 13-58.
        \bibitem{P} Perelman, G. \emph{Two soliton collision for nonlinear Schr\"odinger equations in dimension 1}, Ann. Inst. H.
		Poincaré, Anal. Non Linéaire, Vol. 28,  2011, p 357–384.
        \bibitem{PR} Planchon, F. and Rapha\"el, P. \emph{Existence and Stability of the log–log Blow-up Dynamics for the $L^2$-Critical Nonlinear Schr\"odinger Equation in a Domain}. Annales Henri Poincaré, Vol. 8, No. 6, 2007, pp. 1177-1219.
		\bibitem{RodSS} Rodnianski, I., Schlag, W. and Soffer, A. \emph{ Asymptotic stability of $n$-soliton states of nls}, arXiv Preprint, 2003, arXiv:1001.1627
		 \bibitem{Weinstein} Weinstein, M. I. \emph{Nonlinear Schr\"odinger equations and sharp interpolation estimates}, Comm. Math. Phys., Vol.  87, No. 4, 1982/83, p 567–576.
		\bibitem{Wu} Wu, Y. \emph{Existence of multisoliton solutions of the gravitational Hartree equation in three dimensions}, Trans. Amer. Math. Soc. 379 (2026), 2405-2440.
        \bibitem{Wu2} Wu, Y. \emph{Expansive solutions with prescribed asymptotics of the classical $N$-body problem}, arXiv preprint, 	arXiv:2606.11509 (2026).
		\bibitem{YRZ-Hartree} Yang, K., Roudenko, S. and Zhao Y. \emph{Stable blow‐up dynamics in the $L^2$‐critical and $L^2$‐supercritical generalized Hartree equation}, Studies in Applied Mathematics, Vol. 145, No. 4, 2020, p 647-695.
		
	\end{thebibliography}

	\vspace{1cm}

\end{document}